\documentclass[10pt]{article}
\usepackage{amsmath,amssymb,amsfonts,amsthm,graphicx,enumerate,textcomp,color}
\usepackage[abbrev]{amsrefs}
\usepackage[a4paper,top=1.in, bottom=1.5in, left=.8in, right=.8in, headsep=0.in, centering]{geometry}
\numberwithin{equation}{section}
\usepackage{esint}

\newcommand{\eps}{\varepsilon}

\newtheorem{theorem}{Theorem}[section]
\newtheorem{lemma}{Lemma}[section]

\newtheorem{cor}{Corollary}[section]
\newtheorem{definition}{Definition}[section]

\newcommand{\R}{{\mathbb R}}

\title{Homogenization for the $p$-Laplacian in a $d$-dimensional ball perforated along the unit sphere: the critical case $p=d$}
\author {Peter  V. Gordon
\thanks{Department of Mathematical Sciences,
Kent State University,
 Kent, OH 44242, USA.} %E-mail: {\tt gordon@math.kent.edu}}
\and 
Yuval Peres
\thanks{Beijing Institute of Mathematical Sciences and Applications (BIMSA), Huairou District, Beijing, China.
Corresponding author. E-mail: \texttt{yperes@gmail.com}}}

\date{\today}

\begin{document}

\maketitle

\begin{abstract}
We study a boundary value problem for the $p$-Laplacian in the perforated domain
$B(0,\rho)\setminus\Gamma\subset \mathbb{R}^d$, where $\rho>1$ and $\Gamma$ is the union of many small compact cavities placed near the unit sphere. The cavities are separated at scale $\eps$, asymptotically equidistributed on the sphere, and have cardinality of order $\eps^{1-d}$. The cavities have diameters of order $\alpha(\eps)\eps$, where $\alpha(\eps)\to0$, and their relative $p$-capacity is comparable to the relative p-capacity of a ball  of the same diameter.
The solution is required to equal $1$ on all cavities and $0$ on $\partial B(0,\rho)$.
 We focus on the critical case $p=d>1$.  

We identify the critical scale through the parameter
$\tau=\lim_{\eps\downarrow0}[\eps\log(1/\alpha(\eps))]^{-1}\in[0,\infty]$.
Thus,
$\alpha(\eps)=\exp[-(1+o(1))/(\tau\eps)]$  when $0<\tau<\infty$.  Away from the unit sphere, the solutions converge to $A_*U_\rho$, where $U_\rho(x)=\min\{1,1-\log |x|/\log\rho\}$ is the radial $d$-harmonic potential of the unit ball in $B(0,\rho)$.  The constant $A_*$ equals $0$ when $\tau=0$, equals $1$ when $\tau=\infty$, and is explicit for $0<\tau<\infty$. 
We construct an explicit ansatz that approximates
the solution for sufficiently small $\eps$ in both $L^{\infty}$ and
in terms of $d$-capacity.

 \end{abstract}

\medskip

\noindent {\bf Keywords:}  $p$-Laplacian, $p$-capacity, homogenization, $L^{\infty}$ estimates.

\medskip

\noindent {\bf 2020 Mathematics Subject Classification:}  35J92,  31C45,  35B27,  35B40,  35J20,   35J25.
\medskip

\section{Introduction}
Let $p,d>1$. 
We consider a boundary value problem for the $p$-Laplacian in a perforated domain  $B(0,\rho) \smallsetminus \Gamma \subset \mathbb{R}^d$.  Here    $B(0,\rho)$ is a ball of radius $\rho>1$ and $\Gamma$ is a union of  many tiny compact sets, approximately uniformly distributed near the unit sphere. In our joint work~\cite{GNP} with F. Nazarov, we analyzed the sub-critical case $p<d$. Here we focus   on the   critical case $p=d$.  More precisely, we are interested in the Perron solution of the following boundary value problem:
\begin{eqnarray}\label{eq:2}
\left\{
\begin{array}{lll}
\Delta_p u= {\rm div}( |\nabla u|^{p-2} \nabla u)=0 &\mbox{in} & B(0,\rho) \setminus  \Gamma,\\
u=1 & \mbox{on} & \Gamma,\\
u=  0 & \mbox{on} & \partial B(0,\rho).
\end{array}
\right.
\end{eqnarray}
Here $1<p<\infty$, and $\Gamma$ is constructed as follows.
 Given sufficiently small $\eps>0$, let $S=S(\eps)$ be a finite set of points on the unit sphere $\mathbb{S}^{d-1}\subset \mathbb{R}^d$ such that
 the Euclidean distance between any two points in $S$ is at least $\eps$  (points in $S$ will be referred to as {\bf anchors}).
 Moreover, assume that:
 \begin{description}
\item  [(${\bf  H_1}$)]  The anchors are asymptotically equidistributed, that is
\begin{eqnarray}\label{eq:4}
\eps^{d-1}
 \sum_{s\in S(\eps)} \delta_s \overset{\ast} \rightharpoonup \sigma \mu \: ~ \mbox{as} ~ \eps\to 0,
\end{eqnarray}
\end{description}
where $\mu$ is the uniform probability measure on the unit sphere $\mathbb{S}^{d-1}$
and $\sigma>0$. Examples of anchor sets on $\mathbb{S}^2$ are depicted in Figure \ref{fig:anchor}.

To each anchor $s\in S$ we assign   a compact subset of the closed unit ball $\bar{B}(0,1) \subset \R^d$ which is denoted by $K_s$.
Finally, we introduce a scaling factor $\alpha(\eps)$ that satisfies the following properties:  $0<\alpha=\alpha(\eps)\le \frac{1}{80}$  and  $\alpha(\eps) \to 0$ as $\eps \to 0$. Using these notations we  define:
\begin{eqnarray}\label{eq:1}
\Gamma=\Gamma_\eps:=\underset{s\in S} {\bigcup } (s+\alpha\eps K_s) \,.
\end{eqnarray}

One may expect that, under assumption (${\bf H_1}$), the solutions $u^\eps$
should approach a radial function away from the unit sphere as $\eps\to 0$,
provided the compact sets $K_s$ have uniform effective weights (the precise hypotheses are stated below).
While the structure of such a solution is very simple away from the unit sphere, its behavior near the unit sphere is still quite irregular.  
The main objective of this paper is to extract the asymptotic behavior of such solutions and to obtain their $L^{\infty}$ approximation when $\eps$ is sufficiently small.

\begin{figure}[!ht]
\centering \includegraphics[width=2in]{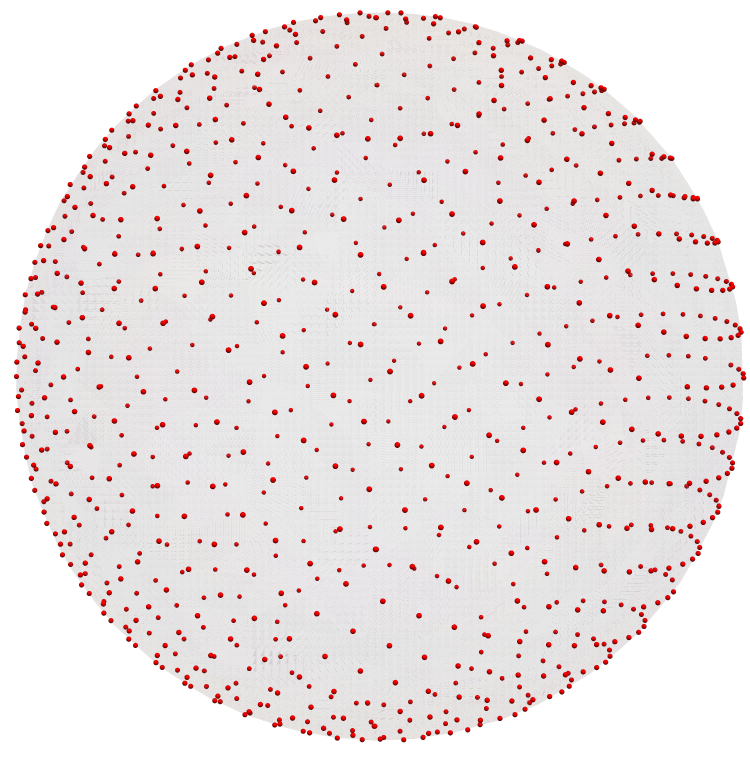} \qquad \centering \includegraphics[width=2in]{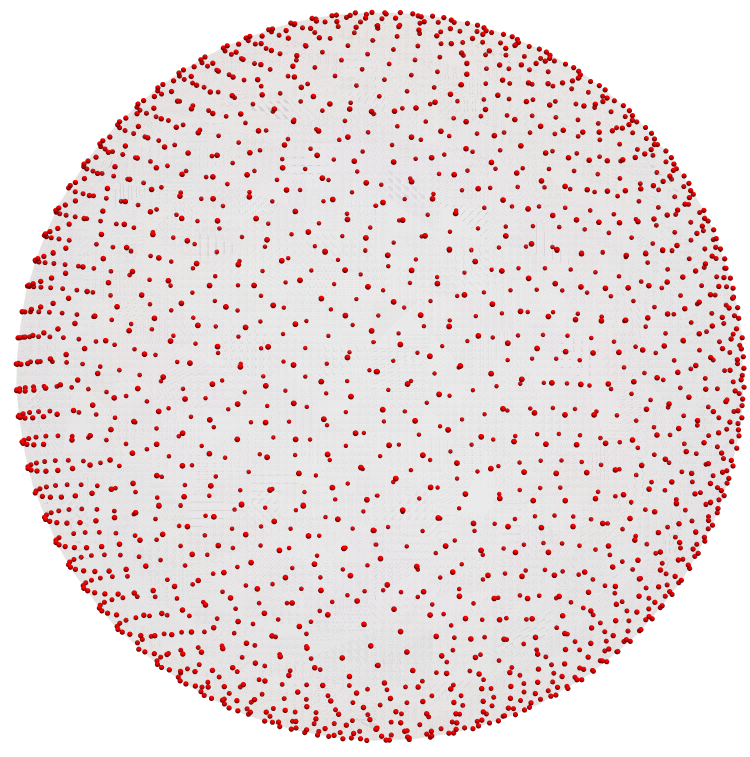} \qquad \includegraphics[width=2in]{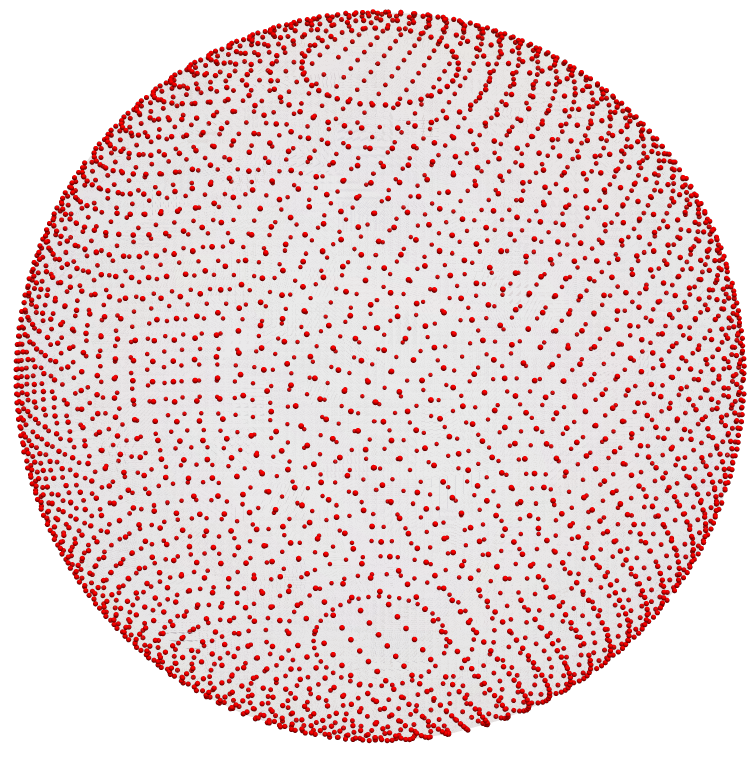}
\caption{Examples of anchor sets $S$ with $500$, $1000$, and $2000$ anchors on the unit sphere $\mathbb{S}^2$, respectively. The anchor locations are indicated by red dots.}
\label{fig:anchor}
\end{figure}

We next recall the definitions of $p$-energy and relative $p$-capacity.

\begin{definition}
Let $\Omega\subset\mathbb R^d$ be open, let $B\subset\Omega$ be Borel, and
let $v\in H_{\rm loc}^{1,p}(\Omega)$ satisfy $\nabla v\in L^p(B)$. The
$p$-energy of $v$ in $B$ is
\begin{eqnarray}
E_p(v,B):=\int_{B} |\nabla v |^p\,dx.
\end{eqnarray}
\end{definition}

\begin{definition} 
The $p$-capacity of a compact set $K\subset \Omega$ relative to a domain $\Omega$ is defined by
\begin{eqnarray}
{\rm cap}_p(K,\Omega):=\inf\{ E_p(\psi,\Omega): ~\psi\ge 1 ~\mbox{on} ~ K, ~ \psi\in C_0^{\infty} (\Omega) \}. 
\end{eqnarray}
\end{definition}
Equivalently, one may use quasicontinuous Sobolev representatives; see
\cite[pp.~27--28 and Theorems~4.4--4.5]{hkm}:
 \begin{eqnarray}
{\rm cap}_p(K,\Omega):=\inf\{ E_p(\psi,\Omega):
~\widetilde\psi\ge 1~p\mbox{-quasieverywhere on }K,
~\psi\in H_0^{1,p}(\Omega) \}.
 \end{eqnarray}
By \cite[Theorems~9.33 and~9.35]{hkm}, with the boundary value on $\Gamma$
understood $p$-quasieverywhere, the Perron solution of \eqref{eq:2} agrees in
$B(0,\rho)\setminus\Gamma$ with the $p$-equilibrium potential. Extended by
$1$ on $\Gamma$, it is the unique capacitary minimizer; uniqueness follows
from strict convexity. In particular,
\begin{eqnarray}
{\rm cap}_p (\Gamma,B(0,\rho))=E_p(u,B(0,\rho)).
\end{eqnarray} 
In this context, $u$ is called the $p$-equilibrium potential of $\Gamma$ in $B(0,\rho)$.

For all sufficiently small $\eps$, the set $\Gamma$ is contained in
$\bar B(0,1+\alpha\eps)\Subset B(0,\rho)$.  Hence monotonicity gives the uniform bound
\begin{eqnarray}
{\rm cap}_p (\Gamma, B(0,\rho))\le {\rm cap}_p (\bar B(0,1+\alpha\eps),B(0,\rho))= \left \{
\begin{array}{ll}
\displaystyle \omega_{d-1} \left |\frac{\gamma}{(1+\alpha\eps)^{-\gamma}-\rho^{-\gamma}}\right |^{p-1} & p\ne d \\
\\
\displaystyle \omega_{d-1} \frac{1}{(\log (\rho/(1+\alpha\eps)))^{d-1}} & p=d,
\end{array}
\right.
\end{eqnarray}
where
\begin{eqnarray}
\gamma:= \frac{d-p}{p-1}
\end{eqnarray}
and $\omega_{d-1}$ is the surface area of the unit sphere in $\mathbb{R}^d$.

The behavior of the solutions of \eqref{eq:2} depends    on the specific values of $p$ and $d$ and on the dependence of $\alpha$ on $\eps$. 
As far as the relation between $p$ and $d$ is concerned, one may distinguish three cases: the subcritical case $1<p<d$, critical case $p=d$, and supercritical case $d<p <\infty$.

The supercritical case $p>d$ is elementary. Assume in this paragraph that
each $K_s$ is nonempty. In this case the Sobolev representative of the
equilibrium potential is H\"older continuous by Morrey's inequality; see,
e.g., \cite[Section~5.6.2]{Evanspde}. Moreover, the following estimate holds.
\begin{eqnarray}
|u(x)-u(y)|\le C |x-y|^{1-d/p} (E_p(u,B(0,\rho)))^{1/p}\le C |x-y|^{1-d/p}
\end{eqnarray}
where $C$ depends only on $p,d$ and $\rho$.

Let
\[
h_\eps:=\sup_{x\in \mathbb S^{d-1}}\min_{s\in S(\eps)} |x-s|
\]
be the mesh size of the anchors.  If the anchors are chosen to be a maximal $\eps$-separated set, then $h_\eps\le \eps$; more generally, the equidistribution hypothesis ${\bf  H_1}$ implies $h_\eps\to0$.  Choose $x^*\in \mathbb S^{d-1}$ such that $u(x^*)=\min_{x\in \mathbb S^{d-1}}u(x)$, and choose an anchor $s^*$ with $|x^*-s^*|\le h_\eps$.  Since $p>d$, every point has positive relative $p$-capacity; hence a $p$-capacity-zero exceptional set is empty.  Thus the continuous Sobolev representative equals $1$ at every point of each nonempty $K_s$.  Choose $k_{s^*}\in K_{s^*}$ and set $y^*=s^*+\alpha\eps k_{s^*}\in\Gamma$.  Then $u(y^*)=1$ and
\begin{eqnarray}
|u(x^*)-u(y^*)|=|1-u(x^*)|\le C (h_\eps+\alpha\eps)^{1-d/p}.
\end{eqnarray}
  Consequently,
  \begin{eqnarray}
  \max_{x\in \mathbb{S}^{d-1} } u(x) - \min_{x\in \mathbb{S}^{d-1} } u(x) \le  C (h_\eps+\alpha\eps)^{1-d/p}
  \end{eqnarray}
  and hence, if $S(\eps)$ is maximal $\eps$-separated,
 \begin{eqnarray}
 u(x)=1+O( \eps^{1-d/p}) \quad \forall x\in \mathbb{S}^{d-1}.
 \end{eqnarray} 
  Thus $u\to1$ uniformly on $\mathbb S^{d-1}$.  The comparison principle in
the inner and outer components then shows that the solution of \eqref{eq:2}
approaches the solution with $\Gamma=\bar B(0,1)$,
  that is
  \begin{eqnarray}
  u^{\eps} (x) \to \left\{
  \begin{array}{ll}
  1 & 0\le |x| \le 1,\\
  \frac{\rho^{-\gamma}-|x|^{-\gamma}}{\rho^{-\gamma}-1} & 1<|x|\le \rho.
  \end{array}
  \right. \quad \mbox{as} \quad \eps\to 0.
  \end{eqnarray}
  The capacity convergence also follows directly from energy comparison.
The monotonicity bound above gives the required limsup. For the
liminf, restrict the energy to $B(0,\rho)\setminus\bar B(0,1)$ and use the
minimum boundary value of $u$ on $\mathbb S^{d-1}$ in the radial variational
problem.  Consequently,
  \begin{eqnarray}
  {\rm cap}_p (\Gamma, B(0,\rho)) \to  {\rm cap}_p (\bar B(0,1),B(0,\rho))= \omega_{d-1} \left (\frac{\gamma}{1-\rho^{-\gamma}}\right)^{p-1}\quad \mbox{as} \quad \eps\to 0.
  \end{eqnarray}
  Apart from nonemptiness, no uniform geometric or capacitary condition on $K_s$ is needed for this conclusion.

 The subcritical case $1<p<d$ is   more delicate.    This case was studied in our recent paper \cite{GNP}  under two additional assumptions:
\begin{description}
 \item [(${\bf  \tilde H_2}$)]   All the sets $K_s$ have the same $p$-capacity, i.e., there is some compact set $K \subseteq  \bar B(0,1)  \subset \R^d$
 such that ${\rm cap}_p(K_s)={\rm cap}_p(K)>0$ for all $\eps>0$ and $s \in S(\eps)$, where
\begin{eqnarray}
{\rm cap}_p(K ):=\inf\left\{ \int_{\mathbb{R}^d} \vert \nabla \psi \vert^p\,dx: ~\psi\in C_0^{\infty} (\mathbb{R}^d), ~ \psi\ge 1 ~\mbox{on} ~ K\right \},
\end{eqnarray}
\end{description}
and
\begin{description}
  \item [(${\bf \tilde  H_3}$)] $\lim_{\eps \to 0} \alpha(\eps)  \eps^{-1/\gamma} =\tau \in [0,\infty]$, where $\gamma:=\frac{d-p}{p-1}$
\end{description}
The assumption ${\bf  \tilde H_2}$ makes the cavities have the same effective capacitary weight.
The assumption ${\bf  \tilde H_3}$ balances the energy contribution from the bulk (the portion of $\bar B(0,\rho)$ away from $\mathbb{S}^{d-1}$) with the contribution from the cavities. If  $\tau=\infty$, then the solution tends to 1 on  $\mathbb{S}^{d-1}$. On the other hand, when $\tau=0$, the solution tends to 0. Thus, ${\bf \tilde  H_3}$
identifies the scaling of the parameter $\alpha$ in which such a transition  takes place.
It was shown in Theorem 1.2  of~\cite{GNP} that, as $\eps \to 0$, the solution of \eqref{eq:2} can be approximated in both $L^{\infty}$ and in the energy space by a relatively simple ansatz.

\bigskip

In view of the results discussed above, the only remaining case is the critical case $d=p$, which is the main subject of this paper. We consider this problem under assumption ${\bf  H_1}$ and two additional assumptions:
\begin{description}
 \item [(${\bf  H_2}$)]   The sets $K_s\subseteq \bar B(0,1)$ have uniformly positive $d$-capacity:
 ${\rm cap}_d(K_s,B(0,3))\ge m_0>0$ for all $\eps>0$ and $s \in S(\eps)$.
 \item [(${\bf  H_3}$)]
 $\displaystyle \lim_{\eps \to 0} \frac{1}{\eps\log(1/\alpha(\eps))}=\tau \in [0,\infty]$.
\end{description}
The assumption ${\bf H_2}$ prescribes uniform nondegeneracy of the cavities. In
the supercritical case no analogous capacitary assumption was needed,
whereas in the subcritical case all compact sets were required to have the
same $p$-capacity. In the critical case, we will show that uniform positivity
suffices and that the specific capacity values play no role in the leading
asymptotics due to conformal invariance of the $p$-energy for $p=d$.
The assumption ${\bf  H_3}$ identifies an exponential critical window.  If $0<\tau<\infty$, it is equivalent to
\begin{eqnarray}
\alpha(\eps)= \exp\left(- \frac{1+o(1)}{\tau \eps}\right).
\end{eqnarray}
The endpoint $\tau=0$ corresponds to smaller cavities, whereas $\tau=\infty$ corresponds to larger cavities (still subject to $\alpha(\eps)\to0$).

From this point onward we work in the critical case $p=d$ and use the shorthand
\begin{eqnarray}\label{def:energycritical}
E(v,D):=E_d(v,D)=\int_D |\nabla v|^d\,dx.
\end{eqnarray}

To formulate our main result, we introduce the $d$-equilibrium potential of the unit ball relative to $B(0,\rho)$, namely, the solution of
\begin{eqnarray}\label{eq:U}
\left\{
\begin{array}{lll}
\Delta_d U_R=0 &\mbox{in} & B(0,R) \setminus \bar{B}(0,1),\\
U_R=1 & \mbox{on} & \bar{B}(0,1),\\
U_R=0 & \mbox{on} & \partial B(0,R) \,,
\end{array}
\right.
\end{eqnarray}
given by
\begin{eqnarray}\label{eq:solU}
U_R(x):=\left\{
\begin{array}{ll}
1& 0\le |x| \le 1 \,,\\
1-\frac{\log |x|}{\log R} & 1<|x|<R \\
0 & |x|\ge R\,.
\end{array}
\right.
\end{eqnarray}
and the ansatz function
\begin{eqnarray}\label{eq:2da3}
u^{\eps}_{A}(x) := A U^{\eps}_{\rho}(x) +(1-A) \sum_{s\in S}U_{\frac{1}{10\alpha}}  \left(\frac{x-s}{\alpha \eps}\right)\,,
\end{eqnarray}
where $U^{\eps}_{\rho}$ is the $d$-equilibrium potential of $\bar B(0,1+\eps)$ relative to $B(0,\rho)$, given by
\begin{eqnarray}\label{eq:solU1}
U^{\eps}_\rho(x):=\left\{
\begin{array}{ll}
1& 0\le |x| \le 1+\eps \,,\\
1-\frac{\log (|x|/(1+\eps))}{\log (\rho/(1+\eps))} & 1+\eps<|x|<\rho\\
0 & |x|\ge \rho\,.
\end{array}
\right.
\end{eqnarray}
For variational and comparison arguments we shall also use the shape-adapted ansatz
\begin{eqnarray}\label{eq:2da3hat}
\widehat u^{\eps}_{A}(x) := A U^{\eps}_{\rho}(x) +(1-A) \sum_{s\in S}V^s_{\frac{1}{10\alpha}}  \left(\frac{x-s}{\alpha \eps}\right)\,.
\end{eqnarray}
Here each $V_R^s$ is the equilibrium potential of $K_s$, defined as
  the Perron solution of the boundary value problem 
\begin{eqnarray}\label{eq:V}
\left\{
\begin{array}{lll}
\Delta_p V_R^s=0 &\mbox{in} & B(0,R)\setminus  K_s,\\
V_R^s=1 & \mbox{on} & K_s,\\
V_R^s(x)= 0 & \mbox{for} & |x| \ge R.
\end{array}
\right.
\end{eqnarray}

Both ansatz functions, $u^{\eps}_{A}$ and $\widehat u^{\eps}_{A}$ have the exact cavity boundary data; the
shape-adapted ansatz follows the geometry of $K_s$ and has its capacitary
energy, whereas $u_A^\eps$ is the explicit radial approximation.  We shall
prove that the two ansatz functions are asymptotically identical in both
$H^{1,d}$ and $L^\infty$.
An example of an ansatz function is depicted in Figure \ref{fig:ansatz}.
\begin{figure}[ht]
\centering \includegraphics[width=\textwidth]{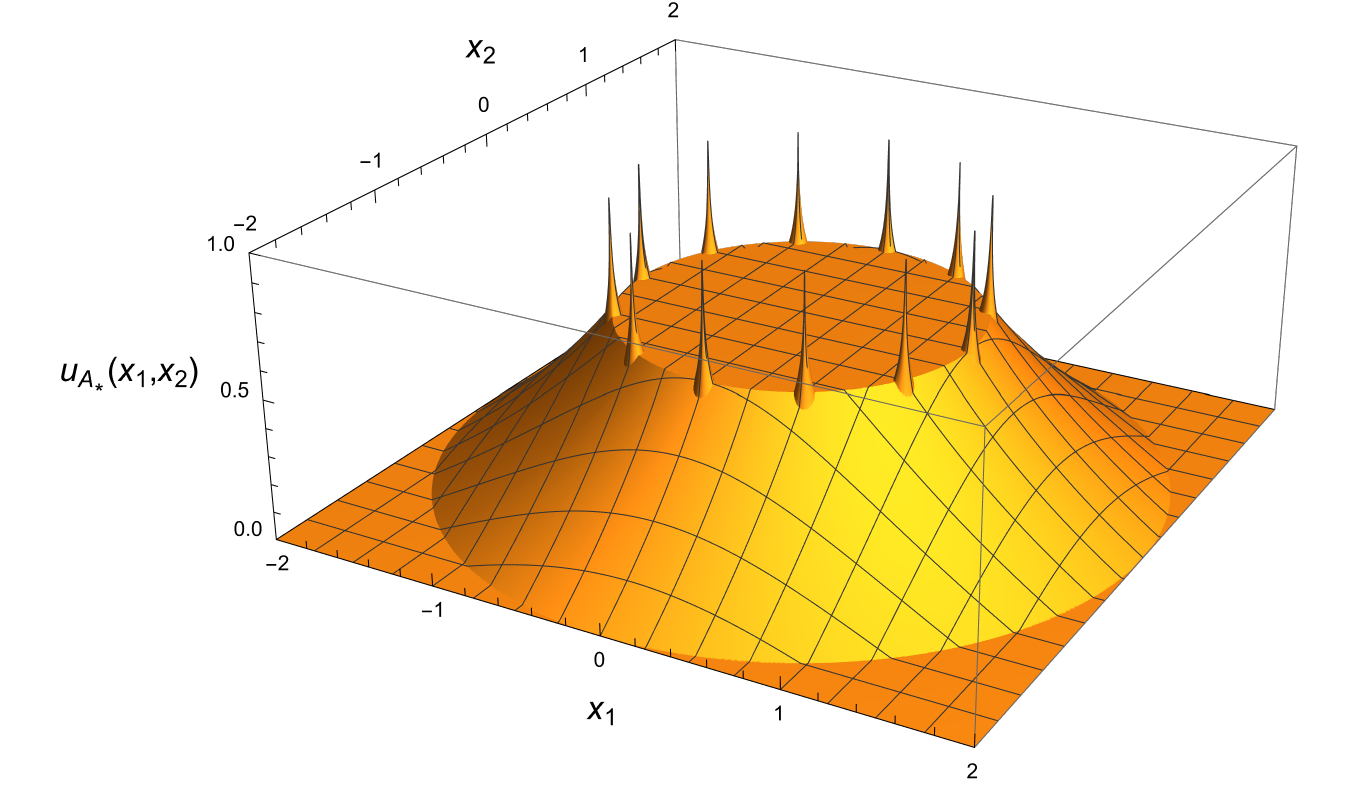}
\caption{The ansatz function $u^\eps_{A_*}$ determined by $12$ equally spaced anchors on the unit circle ($\sigma=2\pi, ~d=p=2$) in a ball of radius $\rho=2$. 
Here the Euclidean separation is $\eps=2\sin(\pi/12)\approx0.518$,
$\tau=1/3$, $\alpha=\exp[-1/(\tau\eps)]\approx3.04\cdot10^{-3}$, and
$A_*\approx0.592$.
 }
\label{fig:ansatz}
\end{figure}

Our main results are as follows.
\begin{theorem}\label{t:1}
Suppose that hypotheses $  {\bf  H_1},{\bf  H_2} , {\bf  H_3}$ hold.
  Then, as $ \eps\to 0$,
\begin{eqnarray} \label{eq:thm1A}
u^\eps(x)\to 
A_* U_\rho(x) 
\end{eqnarray}
uniformly on compact subsets of  $\bar B(0,\rho) \setminus \mathbb{S}^{d-1},$
where
\begin{eqnarray}\label{eq:A*}
A_*=A_*(\tau)=
\left\{
\begin{array}{ll}
\displaystyle \frac{\tau \sigma^\frac{1}{d-1}\log\rho}{1+\tau \sigma^\frac{1}{d-1}\log\rho}, & 0\le \tau<\infty,\\[1.2ex]
1, & \tau=\infty .
\end{array}
\right.
\end{eqnarray}
Furthermore,
\begin{eqnarray} \label{eq:thm1B}
{\rm cap}_d(\Gamma_\eps,B(0,\rho))\to \frac{\omega_{d-1}}{(\log \rho)^{d-1}} A_*^{d-1} .
\end{eqnarray}
\end{theorem}

\begin{theorem} \label{main2}
  Suppose that hypotheses ${\bf  H_1},{\bf  H_2} , {\bf  H_3}$ hold. Then as $\eps\to 0$, we have
\begin{eqnarray}
\Vert u^\eps-\widehat u^\eps_{A_*}\Vert_{{H_0^{1,d}(B(0,\rho))}}\to 0,
\end{eqnarray}
and
\begin{eqnarray}
\Vert u^\eps-\widehat u^\eps_{A_*}\Vert_{L^{\infty}(B(0,\rho))}\to 0 \,.
\end{eqnarray}
Moreover,
\begin{eqnarray}
\Vert \widehat u^\eps_{A_*}-u^\eps_{A_*}\Vert_{{H_0^{1,d}(B(0,\rho))}}
+\Vert \widehat u^\eps_{A_*}-u^\eps_{A_*}\Vert_{L^\infty(B(0,\rho))}\to0.
\end{eqnarray}
Consequently the same two convergence statements hold with the explicit  ansatz $u^\eps_{A_*}$ in place of $\widehat u^\eps_{A_*}$.
\end{theorem}

\smallskip

\medskip
\noindent\textbf{Related work.}
Standard references for the $p$-Laplacian, nonlinear potential theory, and
capacity are \cites{hkm,Lindq,Mazya}; see also the survey \cite{pLrev}.
For general background on homogenization in perforated domains, including
domains with fine-grained boundaries, we refer to \cite{MKhomo}.  The
classical capacitary mechanism by which a family of vanishing holes produces
a nontrivial term in the limiting problem was isolated by Cioranescu and
Murat \cite{CM97}.

The works most closely related to the geometry considered here concern
perforations concentrated near a hypersurface.  Karakhanyan and
Str\"omqvist \cite{KS14} studied a periodic sieve intersecting a hyperplane
and established homogenization for $p<1+d/2$.  They subsequently treated
strictly convex $C^2$ hypersurfaces for $p<1+d/4$ \cite{KS16}.
G\'{o}mez, P\'{e}rez, Podolskii, and Shaposhnikova \cite{GPPS} considered
variational inequalities for the $p$-Laplace operator in periodically
perforated media whose perforations are distributed along a
$(d-1)$-dimensional manifold. At the borderline exponent $p=d$, logarithmic
capacity leads to a qualitatively different scaling. A related critical
$n$-Laplace problem for very thin cavities with a nonlinear boundary
condition was studied by Podol'skii and Shaposhnikova \cite{PS15}. A
boundary-homogenization problem with rapidly alternating nonlinear boundary
conditions was studied by Podolskiy and Shaposhnikova \cite{PS19}; the
geometries and boundary conditions in both works differ from ours.

Our preceding paper \cite{GNP} treats the same spherical arrangement for
$1<p<d$, where the critical scale is algebraic, and proves uniform and
strong Sobolev approximation by an ansatz.  The present paper addresses the
borderline case $p=d$: the critical scale is exponential, and we describe
both the homogenized limit away from the sphere and the near-cavity profile.

\section{Preliminaries}

\subsection{Notation}

In this section we quote several standard results and definitions which will be used in the following sections.
In what follows $\Omega\subset \mathbb{R}^d$ is a bounded open set.

\begin{definition}[{\cite[Chapters~6--7]{hkm}}]
Given an open set $\Omega\subset \mathbb{R}^d$, a continuous function $w\in H_{loc}^{1,p}(\Omega)$ is called $p$-{\bf harmonic} if $\Delta_p w=0$  weakly, i.e., if
$$ \int_\Omega |\nabla w|^{p-2} \nabla w  \cdot \nabla \eta=0 \,, \quad \forall \eta \in C_0^\infty(\Omega)\,. $$
A function $\psi:\Omega \to (-\infty,\infty] $ is $p$-{\bf superharmonic} if it is lower semicontinuous, not identically $\infty$ on any connected component of $\Omega$,
and satisfies the following comparison inequality on compactly contained subdomains $D \subset \Omega$: If $\psi \ge w$ on $\partial D$ and $w\in C(\bar{D})$ is $p$-harmonic in $D$, then $ \psi \ge w$ in $D$.
\end{definition}
\begin{lemma}[Comparison principle, {\cite[Proposition~7.6]{hkm}}]\label{l:comp}
Let $v,w$  be bounded $p$-harmonic
  functions in an open set $\Omega \subset \R^d$.
 If
 $$\limsup_{x \to y} v(x) \le \liminf_{x \to y} w(x)$$
 for all $y \in \partial \Omega$
  then $v \le w$  in $\Omega$.
  \end{lemma}
Next we give the definition of {\bf Perron solutions} from \cite[Chapter~9]{hkm}.
\begin{definition} \label{def:upper}
 Given a domain $\Omega \subset \mathbb{R}^d$ and boundary values $f: M \to \mathbb{R}$, where $\partial \Omega \subset M \subset \Omega^c$,
 the {\em upper class} $U_f^\Omega$ of $f$ consists of $p$-superharmonic functions $\psi:\Omega \to (-\infty,\infty] $ that are bounded below  and satisfy
\begin{eqnarray}\label{upperclass}
\liminf_{x \to y} \psi(x) \ge f(y), \quad \forall y \in \partial \Omega.
\end{eqnarray}
The {\em upper Perron solution} $$h=\bar{H}_f^\Omega: \Omega \cup M \to [-\infty,\infty]$$ of  the boundary value problem
\begin{eqnarray}\label{eq:perron1}
\left\{
\begin{array}{lll}
\Delta_p h=0 &\mbox{in} & \Omega  \, ,\\
h(y)=f(y) & \mbox{for} & y \in M  \\
\end{array}
\right.
\end{eqnarray}
is defined in $\Omega$ by
\begin{eqnarray}\label{upperperron}
\bar{H}_f^\Omega(x):=\inf \{\psi(x) \, :\, \psi \in U_f^{\Omega} \}
\end{eqnarray}
and extended to agree with $f$ in $M$. The {\em lower Perron solution} of \eqref{eq:perron1} is defined by $\underline{H}_f^\Omega:= -\bar{H}_{-f}^\Omega$.

We say that $f$ is {\em resolutive} in $\Omega$ if the upper and lower Perron solutions of \eqref{eq:perron1} coincide; in this case we   refer to both of these simply as the  Perron solution.
\end{definition}

We shall repeatedly use the following variational formulation.  Boundary
conditions on an arbitrary compact set are understood for quasicontinuous
representatives and hold $p$-quasieverywhere.
\begin{lemma}[Capacitary variational principle, {\cite[Theorems~9.33 and~9.35]{hkm}}]\label{lem:variational}
Let $K\Subset\Omega$ be compact and choose $\phi\in C_0^\infty(\Omega)$ with
$\phi=1$ in a neighborhood of $K$.  Set
\begin{eqnarray}\label{eq:admissible}
{\mathcal A}(K,\Omega,\phi)
:=\{v\in H_0^{1,p}(\Omega):v-\phi\in H_0^{1,p}(\Omega\setminus K)\}.
\end{eqnarray}
Then the $p$-equilibrium potential $v_K$ is the unique minimizer of
$E_p(\,\cdot\,,\Omega)$ over this convex class, and
\begin{eqnarray}\label{eq:capmin}
E_p(v_K,\Omega)={\rm cap}_p(K,\Omega).
\end{eqnarray}
It has a quasicontinuous representative equal to $1$ $p$-quasieverywhere on
$K$, vanishes on $\partial\Omega$ in the Sobolev trace sense, and is
$p$-harmonic in $\Omega\setminus K$.  Weak comparison remains valid when the
boundary inequalities on $K$ are imposed $p$-quasieverywhere.
\end{lemma}
\begin{proof}
The cited theorems identify the Perron solution with the capacitary
potential. Existence follows from the direct method, and uniqueness follows
from the strict convexity of $\xi\mapsto|\xi|^p$. For weak comparison,
Theorems~4.4--4.5 of \cite{hkm} justify using the positive part of the
difference as a test function; the conclusion then follows from the strict
monotonicity of $\xi\mapsto|\xi|^{p-2}\xi$.
\end{proof}

In particular, $u^\eps$ and $V_R^s$, extended by $1$ on the corresponding
compact sets and by $0$ outside their outer domains, are Sobolev energy
minimizers.  This convention avoids any regularity assumption on
$\partial K_s$.

\medskip

\subsection{Solutions  in large balls}

For $R>3$, let $V_R^s$ denote the $d$-capacitary potential of $K_s$
relative to $B(0,R)$, extended by $1$ on $K_s$ and by $0$ outside
$B(0,R)$.  Thus $V_R^s$ is $d$-harmonic in $B(0,R)\setminus K_s$ and its
inner boundary value is understood $d$-quasieverywhere.

\begin{lemma} Assume that $R$ is sufficiently large and let $V_R^s$ be as defined above.
Then,
\begin{eqnarray}\label{eq:P12}
\min_{|x|=2}  V_R^s (x)\ge  1- C  \frac{1}{\log R},
\end{eqnarray}
where $C=C(d,K_s)>0$ depends only on $d$ and the compact set $K_s$.
\end{lemma}

\begin{proof} Let $R\gg1$, fix $s\in S$, and set
\begin{eqnarray}\label{eq:P1}
\tilde V_R^s:= 1-V_R^s, \quad \tilde Z_R=1-Z_R,
\end{eqnarray}
where
\begin{eqnarray}
Z_R(x)=\left\{
\begin{array}{ll}
1 & 0\le |x| \le 2,\\
1-\frac{\log (|x|/2)}{\log (R/2)} & 2<|x|<R,\\
0& |x|\ge R,
\end{array}
\right.
\end{eqnarray}
 is the $d$-equilibrium potential of $\bar B(0,2)$ in $B(0,R)$.

We claim that the following two inequalities hold
\begin{eqnarray}\label{eq:P2}
\tilde V_R^s(x) \le \tilde V_3^s(x) \max_{|y|=3} \tilde V_R^s(y), \quad x\in B(0,3)\setminus K_s
\end{eqnarray}
and
\begin{eqnarray}\label{eq:P3}
\tilde V_R^s (y)\le \tilde Z_R(y)+Z_R(y) \max_{|x|=2} \tilde V_R^s(x), \quad y\in B(0,R)\setminus \bar B(0,2).
\end{eqnarray}
Both sides of these inequalities are $d$-harmonic in the indicated domains.
On $K_s$ their quasicontinuous boundary values are zero $d$-quasieverywhere,
so the weak comparison principle from Lemma \ref{lem:variational} applies without
any regularity assumption on $\partial K_s$.  On $\partial B(0,3)$ we have
$\tilde V_3^s=1$, and therefore \eqref{eq:P2} reduces to
$\tilde V_R^s(x) \le \max_{|y|=3} \tilde V_R^s(y)$.
Thus, \eqref{eq:P2} holds. Next we consider \eqref{eq:P3}. When $y\in \partial B(0,2)$ this inequality reads $\tilde V_R^s(y)\le \max_{|x|=2} \tilde V_R^s(x)$ for $|y|=2$.
When $y\in \partial B(0,R)$ \eqref{eq:P3} becomes an equality as $\tilde V_R^s(y)=\tilde Z_R(y)=1$ and $Z_R(y)=0$ when $|y|=R$.

Next observe that inequality \eqref{eq:P2} implies
\begin{eqnarray}\label{eq:P4}
\max_{|x|=2} \tilde V_R^s (x) \le \max_{|x|=2} \tilde V_3^s(x) \max_{|y|=3}  \tilde V_R^s (y) 
\end{eqnarray}
Recall that 
\begin{eqnarray}\label{eq:P5}
\max_{|x|=2} \tilde V_3^s(x)=1-\min_{|x|=2} V_3^s(x)=1-a
\end{eqnarray}
for some $a\in (0,1]$ where $a=a(K_s)$.
Hence,
\begin{eqnarray}\label{eq:P6}
\max_{|x|=2} \tilde V_R^s (x) \le (1-a) \max_{|y|=3}  \tilde V_R^s (y) 
\end{eqnarray}
On the other hand \eqref{eq:P3} gives
\begin{eqnarray}\label{eq:P7}
\max_{|y|=3}  \tilde V_R^s (y)\le \max_{|y|=3}\tilde Z_R(y)+\max_{|y|=3} Z_R(y) \max_{|x|=2} \tilde V_R^s(x)
\end{eqnarray}
Substituting the definition of $Z_R$ into the inequality above we have
\begin{eqnarray}\label{eq:P8}
\max_{|y|=3}  \tilde V_R^s (y)\le \frac{\log(3/2)}{\log(R/2)} +\left( 1-\frac{\log(3/2)}{\log(R/2)} \right)\max_{|x|=2} \tilde V_R^s(x)
\end{eqnarray}
Combining \eqref{eq:P6} and \eqref{eq:P8} we obtain
\begin{eqnarray}\label{eq:P9}
\max_{|y|=3}  \tilde V_R^s (y)\le \frac{\log(3/2)}{\log(R/2)} +(1-a) \left( 1-\frac{\log(3/2)}{\log(R/2)} \right)\max_{|y|=3}  \tilde V_R^s (y) 
\end{eqnarray}
Hence,
\begin{eqnarray}\label{eq:P10}
\max_{|y|=3}  \tilde V_R^s (y)\le \frac{1}{a} \frac{\log(3/2)}{\log(R/2)} .
\end{eqnarray}
Combining \eqref{eq:P6} and \eqref{eq:P10} and using the definition of $V_R^s$  we have
\begin{eqnarray}\label{eq:P11}
1-\min_{|x|=2}  V_R^s (x)=\max_{|x|=2} \tilde V_R^s (x) \le \frac{(1-a)}{a} \frac{\log(3/2)}{\log(R/2)}. 
\end{eqnarray}
Therefore, we have
\begin{eqnarray}\label{eq:P12proof}
\min_{|x|=2}  V_R^s (x)\ge  1- C  \frac{1}{\log R},
\end{eqnarray}
where the constant $C=C(d,K_s)>0$ depends only on $d$ and the compact set $K_s$.
\end{proof}

\begin{lemma}\label{lem:uniform-a} Let $a(s) =\min_{|x|=2} V^s_3(x)$. Then, under the assumption ${\bf  H_2}$, we have
\[
\min_{s\in S} a(s) \ge c_d m_0^\frac{1}{d-1},
\]
where $c_d>0$ depends only on $d$.
\end{lemma}
\begin{proof}
Fix $s\in S$ and write $a=a(s)$.  Since $K_s\subset\bar B(0,1)$,
$V_3^s$ is positive and $d$-harmonic on the fixed annulus
$B(0,5/2)\setminus\bar B(0,3/2)$. Harnack's inequality and the standard
interior gradient estimate for $d$-harmonic functions (see
\cite[Chapters~3--4]{Lindq}) give
\begin{eqnarray}\label{eq:trace-estimate}
\sup_{|x|=2}\bigl(V_3^s(x)+|\nabla V_3^s(x)|\bigr)\le C_d a.
\end{eqnarray}
Set $h(\theta)=a+(1-a)V_3^s(2\theta)$ on $\mathbb S^{d-1}$ and define on
$2\le r\le3$
\begin{eqnarray}
g^s(r\theta)=(3-r)h(\theta).
\end{eqnarray}
Then $g^s=h$ on $|x|=2$, $g^s=0$ on $|x|=3$, and
\eqref{eq:trace-estimate} yields $E(g^s,B(0,3)\setminus\bar B(0,2))\le C_d a^d$.
The function
\begin{eqnarray}
w(x)=\begin{cases}
a+(1-a)V_3^s(x),&x\in B(0,2)\setminus K_s,\\
g^s(x),&x\in B(0,3)\setminus\bar B(0,2),
\end{cases}
\end{eqnarray}
belongs to the admissible class for $K_s$ in $B(0,3)$.  Hence Lemma
\ref{lem:variational} gives
\begin{eqnarray}
E(V_3^s,B(0,3))
&\le&E(w,B(0,3))\\
&\le&(1-a)^dE(V_3^s,B(0,3))+C_d a^d.
\end{eqnarray}
Since $1-(1-a)^d\ge a$ for $0\le a\le1$ and
$E(V_3^s,B(0,3))={\rm cap}_d(K_s,B(0,3))\ge m_0$, we obtain
$a^{d-1}\ge m_0/C_d$, uniformly in $s$.
\end{proof}

Combining the two preceding lemmas gives the following uniform comparison.
For $R\ge R_0(d,m_0)$,
\begin{eqnarray}
\Vert V_R^s-U_R\Vert _{L^{\infty}(B(0,R))} \le \frac{C}{\log R}, \quad \forall s\in S,
\end{eqnarray}
and
\begin{eqnarray}
\left (1-\frac{C}{\log R}\right) E(U_R, B(0,R))\le  E(V_R^s, B(0,R)) \le E(U_R, B(0,R)),
\end{eqnarray}
where $C=C(d,m_0)$.  Indeed, $U_R$ is admissible for $K_s$, and hence
$E(V_R^s)\le E(U_R)$.   Monotonicity of equilibrium potentials yields the pointwise inequality $V_R^s\le U_R$.  The two lemmas imply that
$V_R^s\ge1-C/\log R$ on $B(0,2)$ and, by comparison with the radial
potential of $\bar B(0,2)$, that $V_R^s\ge U_R-C/\log R$ outside that ball.
Finally, $(1-C/\log R)^{-1}V_R^s$ is admissible for $\bar B(0,2)$;
the capacity formula for concentric balls gives the asserted lower energy bound.

\subsection{Energy of an ansatz function}

Recall the radial and shape-adapted ansatz functions from \eqref{eq:2da3} and
\eqref{eq:2da3hat}.  The supports of all bump gradients are pairwise disjoint
and are also disjoint from the support of $\nabla U_\rho^\eps$.  Therefore
\begin{eqnarray} \label{decomp-ener}
E(u_A^\eps,B(0,\rho))&=&A^dE(U_\rho^\eps,B(0,\rho))
 +(1-A)^d |S|q_\eps,\nonumber\\
E(\widehat u_A^\eps,B(0,\rho))&=&A^dE(U_\rho^\eps,B(0,\rho))
 +(1-A)^d\sum_{s\in S}\kappa_s,
\end{eqnarray}
where
\begin{eqnarray}\label{def:qkappa}
q_\eps:=E\left(U_{\frac1{10\alpha}},B\left(0,\frac1{10\alpha}\right)\right),
\qquad
\kappa_s:={\rm cap}_d\left(K_s,B\left(0,\frac1{10\alpha}\right)\right).
\end{eqnarray}
A direct computation and the comparison in the preceding subsection give,
with $L_\eps:=\log(1/(10\alpha))$,
\begin{eqnarray}
E(U_\rho^\eps,B(0,\rho))&=&\frac{\omega_{d-1}}{(\log (\rho/(1+\eps)))^{d-1}}
=\frac{\omega_{d-1}}{(\log \rho)^{d-1}}(1+O(\eps)),\nonumber\\
q_\eps&=&\frac{\omega_{d-1}}{L_\eps^{d-1}},\nonumber\\
\left(1-\frac{C}{L_\eps}\right)q_\eps&\le&\kappa_s\le q_\eps
\qquad(s\in S).
\end{eqnarray}
Since
\begin{eqnarray}
\eps^{d-1} |S| \to \sigma \quad \mbox{as} \quad \eps\to 0
\end{eqnarray}
and, for $0\le \tau<\infty$, $L_\eps^{-1}/\eps\to \tau$, we have
\begin{eqnarray}\label{limener}
E(u_A^\eps,B(0,\rho)),\ E(\widehat u_A^\eps,B(0,\rho))
 \longrightarrow \varphi_\tau(A),
\end{eqnarray}
uniformly for $A\in[0,1]$ when $\tau<\infty$, where
\begin{eqnarray}\label{def:varphi}
\varphi_\tau(A)=\omega_{d-1}\left (\frac{A^d}{(\log \rho)^{d-1}}+(1-A)^d \sigma \tau^{d-1} \right), \qquad 0\le \tau<\infty.
\end{eqnarray}
For $\tau=\infty$, we use the extended convention $\varphi_\infty(1)=\omega_{d-1}/(\log\rho)^{d-1}$ and $\varphi_\infty(A)=\infty$ for $0\le A<1$.
For $0<\tau<\infty$, the minimum of $\varphi_\tau$ is attained at
\begin{eqnarray}
A_*=\frac{\tau \sigma^\frac{1}{d-1}\log\rho}{1+\tau \sigma^\frac{1}{d-1}\log\rho}.
\end{eqnarray}
Also,
\begin{eqnarray}\label{phimin}
\varphi_\tau(A_*)=\omega_{d-1}\left(  \frac{\tau  \sigma ^\frac{1}{d-1}}{1+\tau \sigma^{\frac{1}{d-1}} \log \rho} \right)^{d-1}
=\frac{\omega_{d-1}}{(\log\rho)^{d-1}}A_*^{d-1}.
\end{eqnarray}
The final identity also holds at $\tau=0$ and $\tau=\infty$ under the
extended convention above.

\begin{lemma}\label{lem:ansatz-comparison}
For $A=A_*$, the two ansatz functions satisfy
\begin{eqnarray}\label{eq:ansatz-comparison}
\|\widehat u_{A_*}^\eps-u_{A_*}^\eps\|_{H_0^{1,d}(B(0,\rho))}
+\|\widehat u_{A_*}^\eps-u_{A_*}^\eps\|_{L^\infty(B(0,\rho))}\longrightarrow0.
\end{eqnarray}
\end{lemma}
\begin{proof}
The uniform large-ball comparison gives
\begin{eqnarray}
\|\widehat u_{A}^\eps-u_A^\eps\|_\infty\le \frac{C(1-A)}{L_\eps},
\end{eqnarray}
because the bump supports are disjoint.  For the gradients, $U_R$ is an
admissible competitor for the $K_s$-potential $V_R^s$, and
$(U_R+V_R^s)/2$ is admissible as well.  Clarkson's inequality
\cites{Clarkson,Boas} and minimality
therefore give
\begin{eqnarray}\label{eq:single-bump-gradient}
\|\nabla U_R-\nabla V_R^s\|_{L^d(B(0,R))}^d
\le C_d\bigl(E(U_R)-E(V_R^s)\bigr)\le \frac{C}{(\log R)^d}.
\end{eqnarray}
Scaling does not change $d$-energy.  Summing over the disjoint supports, for
$\tau<\infty$ we obtain
\begin{eqnarray}
\|\nabla\widehat u_A^\eps-\nabla u_A^\eps\|_d^d
\le C(1-A)^d|S|L_\eps^{-d}
=O\!\left(\eps\left(\frac1{\eps L_\eps}\right)^d\right)=o(1).
\end{eqnarray}
For $\tau=\infty$ we have $A_*=1$, so the two ansatz functions coincide.
Poincar\'e's inequality completes the proof.
\end{proof}

\subsection{Energy estimate on truncated cones}

\begin{definition}
Let  $y \in {\mathbb S}^{d-1}$ and let $Q_{\delta}(y):=B(y,\delta) \cap {\mathbb S}^{d-1}$ be the open  spherical cap of Euclidean radius $\delta$ centered at $y$.
We define the spherical cone
\begin{eqnarray}\label{a1}
\Lambda_{\delta}(y):=\{ r q: q\in Q_{\delta}(y)\; \, \text{\rm and} \; \, \rho>r>0\}.
\end{eqnarray}
\end{definition}

\begin{lemma}\label{l:holder}
Let $Q$ be an open set  on the unit sphere,  and for $\rho>R>1$ set
\begin{eqnarray}\label{eqh:3}
\Lambda(Q,R):=\{ r q: q\in Q\; \, \text{\rm and} \; \, R<r<\rho\}\,.
\end{eqnarray}
Suppose that a nonnegative function $v\in H^{1,d}(\Lambda(Q,R)) \cap C(\overline{\Lambda(Q,R)})$ satisfies
 \begin{eqnarray}\label{eqh:7a}
 v(x)=0 \quad  \mbox{at}  \quad |x|=\rho \,,
 \end{eqnarray}
 and
 \begin{eqnarray}\label{eqh:7b}
 v(Ry)\ge \tilde A>0, \quad \forall y\in Q\,.
 \end{eqnarray}
Then,
\begin{eqnarray}\label{eqh:7}
E(v,\Lambda(Q,R)) \ge E(v_{\tilde A},\Lambda(Q,R)) \,,
\end{eqnarray}
where  
\begin{eqnarray}\label{eqh:7x}
v_{\tilde A}(x)=\tilde A \left(1-\frac{\log(|x|/R)}{\log(\rho/R)}\right)\quad \forall x \in \Lambda(Q,R).
\end{eqnarray}
\end{lemma}

\begin{proof}
Assume \eqref{eqh:7a} and \eqref{eqh:7b} hold. Cover $Q$ by countably many
relatively compact spherical coordinate patches. On each such patch, the
polar-coordinate map $(r,z)\mapsto r\Phi(z)$ is bi-Lipschitz. Composition
with this map preserves $H^{1,d}$, and the ACL theorem
\cite[Theorem~4.9.2]{EG92}, followed by Fubini's theorem, shows that for
$\mu$-almost every $y\in Q$ the function $r\mapsto v(ry)$ is absolutely
continuous on $(R,\rho)$ and
$|\partial_r v(ry)|\le|\nabla v(ry)|$ for almost every $r$. Thus, for
$\mu$-almost every $y\in Q$, H{\"o}lder's inequality yields
\begin{eqnarray}\label{eqh:1}
&& \tilde A\le \int_R^{\rho} \left | \nabla v(ry)\right| dr=\int_R^{\rho}  \left | \nabla v(ry)\right| r^\frac{d-1}{d} r^{\frac{1-d}{d}} \, dr\le \left(\int_R^{\rho}  \left | \nabla v(ry)\right|^d r^{d-1} \, dr\right)^\frac1d
\left( \int_R^{\rho} \frac{dr}{r}\right)^\frac{d-1}{d}\nonumber \\
&&=\left(\log(\rho/R)\right)^\frac{d-1}{d}  \left(\int_R^{\rho}  \left | \nabla v(ry)\right|^d r^{d-1}dr\right)^\frac1d.
\end{eqnarray}
Taking the $d$th power and rearranging terms, we obtain
\begin{eqnarray}\label{eqh:2}
\frac{\tilde A^d}{\left(\log(\rho/R)\right)^{d-1} }\le \int_R^{\rho}  \left | \nabla v(ry)\right|^d r^{d-1}dr.
\end{eqnarray}
Let   $\omega_{d-1}$ denote the area of the unit sphere.
Integrating \eqref{eqh:2} over $Q$, we obtain
\begin{eqnarray}\label{eqh:4}
\frac{\tilde A^d \omega_{d-1} \mu (Q)  }{\left(\log(\rho/R)\right)^{d-1} } \le \omega_{d-1} \int_Q \int_R^{\rho}  \left | \nabla v(ry)\right|^d r^{d-1}\, dr \, d\mu(y)=\int_{\Lambda(Q,R)}   \left | \nabla v(x) \right|^d \, dx =E(v,\Lambda(Q,R))\,.
\end{eqnarray}

By the condition for equality in H{\"o}lder's inequality, the second inequality in \eqref{eqh:1} is an equality if for some constant $C$,
\begin{eqnarray}\label{eqh:5}
 |\nabla v(ry)| =\frac{C}{r} \quad \mbox{\rm for a.e.} \; R<r<\rho \,.
\end{eqnarray}
The function $v_{\tilde A}(x)$  defined by \eqref{eqh:7x} satisfies  \eqref{eqh:5} and the boundary condition $v_{\tilde A}(Ry) =\tilde A$ for all $y \in Q$. Since
$$ |\nabla v_{\tilde A}(ry)|= -\frac{d}{dr}  v_{\tilde A}(ry)
$$
for $y \in \partial B(0,1)$, the first inequality in  \eqref{eqh:1} is also an equality if $v=v_{\tilde A}$. Therefore,
\begin{eqnarray}\label{eqh:6}
E(v_{\tilde A},\Lambda(Q,R)) = \frac{\tilde A^d \omega_{d-1} \mu (Q)  }{\left(\log(\rho/R)\right)^{d-1} }  \,.
\end{eqnarray}
Combining \eqref{eqh:4} and \eqref{eqh:6}, we obtain   \eqref{eqh:7}.
\end{proof}

\begin{lemma}[Uniform cap counts]\label{lem:uniform-caps}
For every fixed $r\in(0,2)$, hypothesis ${\bf  H_1} $ implies
\begin{eqnarray}\label{eq:uniform-caps}
\sup_{y\in\mathbb S^{d-1}}
\left|\eps^{d-1}|S(\eps)\cap Q_r(y)|-\sigma\mu(Q_r)\right|\longrightarrow0.
\end{eqnarray}
More generally, the same conclusion holds if $r$ is replaced by a scalar
sequence $r_\eps\to r$. If the perturbation depends on the center $y$, the
convergence to $r$ is assumed uniform in $y$.
\end{lemma}
\begin{proof}
If uniform convergence failed, one could choose $\eps_j\downarrow0$ and
$y_j\to y$ for which the discrepancy stays bounded away from zero.  For
every $\eta>0$ and all large $j$,
$Q_{r-\eta}(y)\subset Q_r(y_j)\subset Q_{r+\eta}(y)$.  Apply weak convergence
in ${\bf  H_1} $ to the two fixed caps and then let $\eta\downarrow0$; their
boundaries have zero $\mu$-measure. The assertions for varying radii follow
from the same sandwich argument (uniformly in $y$ in the center-dependent
case).
\end{proof}

\begin{cor} \label{cor:en}
For every fixed $\delta\in(0,1/2)$ and $\tau \in [0,\infty]$, we have
\begin{eqnarray}\label{c13b}
E(\widehat u_A^\eps,\Lambda_\delta(y)) \to \mu(Q_\delta) \varphi_\tau(A)
\quad \mbox{as} \quad \eps \to 0 \,,
\end{eqnarray}
and the convergence is uniform in $A \in [0,1]$ and in $y \in {\mathbb S}^{d-1}$, provided that $\tau<\infty$.
Moreover, for sufficiently small $\eps$ (depending on $\delta$),
\begin{eqnarray}\label{eq:bady3}
E(\widehat u_A^\eps, \Lambda_{\delta}(y))
\ge (1-\delta)^d E(\widehat u_{A_*}^\eps, \Lambda_{\delta}(y))
\quad(A\in[0,1],\ y\in\mathbb S^{d-1}).
\end{eqnarray}
\end{cor}
\begin{proof}
The disjoint supports give
\begin{eqnarray}\label{decomp-ener2}
&&A^dE(U_\rho^\eps,\Lambda_\delta(y))
+(1-A)^d\sum_{s\in S\cap Q_{\delta-\eps}(y)}\kappa_s\nonumber\\
&&\qquad\le E(\widehat u_A^\eps,\Lambda_\delta(y))\nonumber\\
&&\qquad\le A^dE(U_\rho^\eps,\Lambda_\delta(y))
+(1-A)^d\sum_{s\in S\cap Q_{\delta+\eps}(y)}\kappa_s.
\end{eqnarray}
The bulk potential is radial, Lemma \ref{lem:uniform-caps} controls the two
anchor counts uniformly in $y$, and \eqref{def:qkappa} gives
$\kappa_s=q_\eps(1+O(L_\eps^{-1}))$ uniformly in $s$.  These facts prove
\eqref{c13b}.

If $0<\tau<\infty$, \eqref{eq:bady3} follows from uniform convergence and
the fact that $\varphi_\tau$ is minimized at $A_*$.  If $\tau=0$, then
$A_*=0$.  For $A\le\delta$, the bump term in \eqref{decomp-ener2} is at
least $(1-\delta)^d$ times the bump term for $A=0$; for $A>\delta$, the
positive bulk term dominates the vanishing right-hand side when $\eps$ is
small.  Finally, if $\tau=\infty$, then $A_*=1$.  The bulk term proves the
claim when $A\ge1-\delta$, while for $A<1-\delta$ the bump term diverges
uniformly as $\eps\downarrow0$.
\end{proof}

\section{Bounds on oscillations and energy estimates} \label{heart}
The next three local lemmas are \cite[Lemmas~3.1--3.3]{GNP}.  Their proofs
use only local $p$-harmonic estimates and the Besicovitch covering theorem,
not the restriction $p<d$ imposed elsewhere in that paper; hence they remain
valid here with $p=d$.  In what follows, positive constants depend only on $d$.
\begin{lemma}[{\cite[Lemma~3.1]{GNP}}] \label{l:1}
Suppose that for some $r>0$ and $z \in B(0,\rho)$ we have
\begin{eqnarray}\label{a4}
B(z,5r/4)\subset B(0,\rho) \setminus \Gamma
\end{eqnarray}
and for some
$\lambda>0, \beta \ge 0$, the solution $u$ of \eqref{eq:2} satisfies
\begin{eqnarray}\label{a3}
E(u, B(z,5r/4)) \le \lambda^d r^{d-1-\beta} \,.
\end{eqnarray}
Then,
\begin{eqnarray}\label{a5}
\underset{ B(z,r)} {\rm osc}~ u  \le C_1 \lambda r^\frac{d-1-\beta}{d} \,,
\end{eqnarray}
where
\begin{eqnarray}\label{a5a}
\underset{ D} {\rm osc}~ u:=\sup_{D} u-\inf_{D} u
\end{eqnarray}
stands for the  oscillation over the set $D$.
\end{lemma}

\smallskip

\begin{definition}\label{defbad}
Fix $\beta =\frac{d-1}{2d}$.  Given $\delta<1/20$ and $\eps<\delta/20$, an anchor $s \in S$ will be called  a {\bf good anchor} if
\begin{eqnarray}
E(u,\Lambda_{\zeta}(s))\le \zeta^{d-1-\beta},\qquad \forall \zeta\in[\eps,\delta] \,.
\end{eqnarray}
Otherwise, $s$ will be called a {\bf bad anchor}.
\end{definition}

\begin{lemma}[{\cite[Lemma~3.2]{GNP}}]\label{l:2}

Fix $\eps,\delta,\beta$ as in Definition \ref{defbad}, and assume additionally
that
\[
\delta<\frac{\rho-1}{4}.
\]
Suppose $s$ is a good anchor and let
\begin{eqnarray} \label{eq:star}
\Lambda_{\eps}^*(s):= \Bigl\{ x\in \Lambda_{\eps/2}(s)\setminus B\left(s,\frac{\eps}{10}\right) : 1-\delta \le \vert x \vert \le 1+\delta \Bigr\} \,.
\end{eqnarray}
Then there exists $C_5$ such that
\begin{eqnarray}
\underset{\Lambda_{\eps}^*(s) } {\rm osc} ~ u \le C_5 \delta^\frac{d-1-\beta}{d} \,.
\end{eqnarray}
\end{lemma}

\begin{lemma}[{\cite[Lemma~3.3]{GNP}}]\label{l:3}
Fix $\eps$, $\delta,$ $\beta$ as in Definition \ref{defbad}, and let $y$ be a point on the unit sphere   ${\mathbb S}^{d-1}$.  Suppose that
\begin{equation} \label{energy3}
 E(u,\Lambda_{2\delta}(y)) \le M\delta^{d-1}
\end{equation}
for some $M>0$. Then the set $S_\beta$ of bad anchors satisfies
$|S_\beta \cap Q_\delta(y)| \le C_7 M\delta^\beta (\delta/\eps)^{d-1}$,
where $C_7=C_7(d)$.
\end{lemma}

\section{A local energy lower bound}

\begin{lemma}\label{l:4}
 Let $\beta:=(d-1)/(2d)$. There exist
$\delta_0<\min\{1/20,(\rho-1)/4\}$ and $C>0$, depending only on
$d,\sigma,m_0,$ and $\rho$, such that for every $\delta\in(0,\delta_0)$
there exists $\eps_0=\eps_0(\delta)>0$ with the following property. For each
$\eps\in(0,\eps_0)$ and $y\in\mathbb S^{d-1}$ such that $u=u^\eps$ satisfies
\begin{eqnarray}\label{eq4:1}
E(u,\Lambda_{2\delta}(y)) \le \delta^{-\beta/2} \delta^{d-1} \,,
\end{eqnarray}
there exists some $A=A(y,\delta,\eps)$ such that
\begin{eqnarray}\label{eq4:5}
|u(x)-A| \le C \delta^\beta, \quad \forall x\in \Lambda_{\delta}(y):  |x|=1\pm\delta \,,
\end{eqnarray}
and
\begin{eqnarray}\label{eq4:2}
E(u,\Lambda_{\delta}(y)) \ge (1- 3\delta^{\beta/3}) E(\widehat u^{\eps}_{A},\Lambda_{\delta}(y))\,.
\end{eqnarray}
\end{lemma}
\begin{proof}
Since $\beta=(d-1)/(2d)$,
\[
\frac{d-1-\beta}{d}-\beta=\frac{(d-1)^2}{2d^2}>0.
\]
Thus $(d-1-\beta)/d>\beta$, a fact used repeatedly below.
 
Also, recall the following consequence of the definition \eqref{eq:2da3hat} of $\widehat u_A^\eps$: for every fixed $\delta>0$ and sufficiently small $\eps$,
\begin{eqnarray} \label{decomp-ener2.5}
&& A^d E(U_\rho^\eps, \Lambda_{\delta}(y))+(1-A)^d\sum_{s\in S\cap Q_{\delta-\eps}(y)}\kappa_s \nonumber\\
&&\le E(\widehat u_A^\eps, \Lambda_{\delta}(y))\le A^d E(U_\rho^\eps, \Lambda_{\delta}(y)) +(1-A)^d \sum_{s \in S \cap Q_{\delta+\eps}(y)} \kappa_s \, .
\end{eqnarray}
Here
\begin{eqnarray} \label{decomp-ener3.5}
\kappa_s:=  {\rm cap}_d\left(s+\alpha \eps K_s, B\left(s,\frac{\eps}{10}\right)\right)
= {\rm cap}_d\left(K_s, B\left(0,\frac{1}{10\alpha}\right)\right) \,.
\end{eqnarray}
The estimates in the subsection on solutions in large balls imply, uniformly in $s$,
\begin{eqnarray}\label{kappa-uniform}
\left(1-\frac{C}{\log(1/(10\alpha))}\right)\frac{\omega_{d-1}}{(\log(1/(10\alpha)))^{d-1}}
\le \kappa_s \le
\frac{\omega_{d-1}}{(\log(1/(10\alpha)))^{d-1}}.
\end{eqnarray}
\medskip

Cover each of the two caps
$\Lambda_\delta(y)\cap\partial B(0,1\pm\delta)$ by a number bounded in terms
of $d$ of balls of radius $c_d\delta$, chosen so that their $5/4$ dilates lie
in $\Lambda_{2\delta}(y)\setminus\Gamma$.  Applying Lemma \ref{l:1} in this
finite overlapping chain and using \eqref{eq4:1} gives
\begin{eqnarray}\label{eq4:3}
\underset{  \Lambda_{\delta}(y)\cap \partial B(0,1+\delta)} {\rm osc}~ u \le C\delta^\frac{d-1-\beta}{d} \le C\delta^{\beta} \,.
\end{eqnarray}

 Similarly,
\begin{eqnarray}\label{eq4:31}
\underset{  \Lambda_{\delta}(y)\cap \partial B(0,1-\delta)} {\rm osc}~ u \le   C\delta^{\beta} \,.
\end{eqnarray}

  Hence,
\begin{eqnarray}\label{eq4:4}
u(x)=A_{\pm}+ O(\delta^\beta), \quad \forall x\in \Lambda_{\delta}(y)\ \mbox{with } |x|=1\pm\delta,
\end{eqnarray}
where $A_+=A_+(y)$ and $A_-=A_-(y)$ are two constants.
In view of Lemma \ref{l:3} with $M=\delta^{-\beta/2}$, if $\delta_0>0$ is small enough and $\delta \in (0,\delta_0)$, then for sufficiently small $\eps>0$, there exists a good anchor in $Q_{\delta}(y)$, and therefore by Lemma \ref{l:2} and \eqref{eq4:4}, we have $|A_+-A_-|\le C\delta^\beta.$
Thus, for some $C_{\mathrm{osc}}>0$,
\begin{eqnarray}\label{eq4:51}
|u(x)-A_+| \le C_{\mathrm{osc}}\delta^\beta,
\quad \forall x\in \Lambda_{\delta}(y): |x|=1\pm\delta \,.
\end{eqnarray}

\medskip
Denote by $Q_{\delta-\eps}^{\sharp}(y)=(S\setminus S_\beta)\cap
Q_{\delta-\eps}(y)$ the set of good anchors in $Q_{\delta-\eps}(y)$.
Lemma \ref{l:2} and \eqref{eq4:51} imply that there exists
$C_{\mathrm{up}}>0$ such that, for every
$s\in Q_{\delta-\eps}^{\sharp}(y)$,
\begin{eqnarray}\label{eq4:6}
u(x)\le A_1:=\min\{1,A_++C_{\mathrm{up}}\delta^\beta\},
\qquad \forall x\in\partial B\left(s,\frac{\eps}{10}\right).
\end{eqnarray}
Define
\[
A:=
\begin{cases}
A_1, & A_+>1/2,\\
\max\{0,A_+-C_{\mathrm{osc}}\delta^\beta\}, & A_+\le1/2.
\end{cases}
\]
Decrease $\delta_0$, if necessary, so that
$C_{\mathrm{up}}\delta^\beta<1/4$. Then $A=1$ if and only if $A_1=1$.
If $A_+\le1/2$, then $1-A\ge1/2$ and
\[
A_1-A\le(C_{\mathrm{up}}+C_{\mathrm{osc}})\delta^\beta.
\]
Consequently,
\[
1-A_1\ge(1-A)\bigl(1-2(C_{\mathrm{up}}+C_{\mathrm{osc}})
\delta^\beta\bigr).
\]
If $A_+>1/2$, then $A=A_1$. Since
$\delta^\beta=o(\delta^{\beta/3})$, after decreasing $\delta_0$ once more,
\begin{eqnarray}\label{compa}
(1-A_1)^d\ge(1-\delta^{\beta/3})(1-A)^d
\end{eqnarray}
for all values of $A_+$.
 
 Recall the definition of $\kappa_s$ in \eqref{decomp-ener3.5} and the uniform estimate \eqref{kappa-uniform}.

Now that $A$ has been chosen, our next goal will be to prove that if $\delta$ is small enough, then the inequality
\begin{eqnarray}\label{eq4:9}
E\left(u, \Lambda_{\delta}(y)\cap B(0,1+\delta) \right) \ge (1-2\delta^{\beta/3}) (1-A)^d \sum_{s \in S \cap Q_{\delta+\eps}(y)} \kappa_s
\end{eqnarray}
holds for  sufficiently small $\eps$. Note that if $A_1=1$ then also $A=1$,
 so the right-hand side of \eqref{eq4:9} vanishes, and the inequality certainly holds; thus we need only prove \eqref{eq4:9} when $A_1<1$.

For $s\in Q_{\delta-\eps}^{\sharp}(y)$, define
\[
u_s(x):=\max\left\{0,\frac{u(x)-A_1}{1-A_1}\right\}.
\]
The function $u_s$ has zero Sobolev trace on $\partial B(s,\eps/10)$ by
\eqref{eq4:6} and equals $1$ $d$-quasieverywhere on
$s+\alpha\eps K_s$. It is therefore
admissible for the relative capacity in Lemma \ref{lem:variational}. Hence,
$$\forall s \in Q_{\delta-\eps}^{\sharp}(y), \qquad  \kappa_s\le E \Bigl(u_s, \, B\left(s,\frac{\eps}{10}\right)\Bigr)\le
(1-A_1)^{-d} E\Bigl(u, B\left(s,\frac{\eps}{10}\right)\Bigr) \,.$$

Therefore,
\begin{eqnarray}\label{eq4:8}
E\left(u,  \Lambda_{\delta}(y)\cap B(0,1+\delta) \right) \ge
\sum_{s\in  Q_{\delta-\eps}^{\sharp}(y)} E\left(u, B\left(s,\frac{\eps}{10}\right)\right)  \ge (1-A_1)^{d}\sum_{s\in  Q_{\delta-\eps}^{\sharp}(y)}    \kappa_s\, .
\end{eqnarray}

Next, we will compare the right-hand sides of \eqref{eq4:8} and \eqref{eq4:9}.
Lemma \ref{lem:uniform-caps} implies, uniformly in $y$, that for $\eps$ sufficiently small,
\begin{eqnarray}\label{eq:kid}
|S \cap Q_{\delta}(y)| \ge (1-\delta)\sigma \eps^{1-d} \mu( Q_{\delta}) \ge C \sigma  (\delta/\eps)^{d-1} \,.
\end{eqnarray}
 Invoking Lemma \ref{l:3} with $M=\delta^{-\beta/2}$ yields that
\begin{eqnarray}\label{eq:kid1}
|S_\beta \cap Q_{\delta}(y)| \le C  \delta^{\beta/2} (\delta/\eps)^{d-1}\,.
\end{eqnarray}
The $\mu$-measure of the shell
$Q_{\delta+2\eps}(y)\setminus Q_{\delta-2\eps}(y)$ is
$O(\eps\delta^{d-2})$. Since the caps of radius $\eps/2$ centered at the
anchors are pairwise disjoint and contained in this larger shell, the number
of anchors in $Q_{\delta+\eps}(y)\setminus Q_{\delta-\eps}(y)$ is
\[
O\bigl((\delta/\eps)^{d-2}\bigr).
\]
By decreasing $\eps_0(\delta)$, if necessary, we may assume
$\eps\le\delta^{1+\beta/2}$. Hence
\[
O\bigl((\delta/\eps)^{d-2}\bigr)
\le C\delta^{\beta/2}(\delta/\eps)^{d-1}.
\]
Thus, for sufficiently small $\eps$,
\begin{eqnarray}\label{eq:kid3}
 \Bigl|\Bigl(S \cap  Q_{\delta+\eps}(y)\Bigr) \setminus  Q_{\delta-\eps}^{\sharp}(y)\Bigr|
 &\le & \Bigl| S_\beta \cap Q_\delta(y)\Bigr| + \Bigl|S \cap \bigl(Q_{\delta+\eps}(y) \setminus  Q_{\delta-\eps}(y)\bigr)\Bigr| \\ \nonumber
  &\le& 2C_7  \delta^{\beta/2} (\delta/\eps)^{d-1} \le  C\sigma\delta^{\beta/3}  (\delta/\eps)^{d-1} \,,
 \end{eqnarray}
 where the rightmost inequality assumes that  $\delta $ is small enough so that $\delta^{\beta/6} \le C \sigma$.

Choose $\delta_0$ smaller, if necessary, so that the fractions of bad
anchors and shell anchors in \eqref{eq:kid3} are each at most
$\frac12\delta^{\beta/3}$ of the lower count in \eqref{eq:kid}.  Using also
the uniform comparability of the weights $\kappa_s$ in
\eqref{kappa-uniform}, and then taking $\eps$ sufficiently small, we obtain
\begin{eqnarray}\label{eq:kid4}
\sum \Bigl\{ \kappa_s \,: \, s\in \Bigl(S \cap  Q_{\delta+\eps}(y)\Bigr) \setminus  Q_{\delta-\eps}^{\sharp}(y)\Bigr\} \le
 \delta^{\beta/3}\sum \Bigl\{ \kappa_s \,: \, s\in S \cap  Q_{\delta+\eps}(y)  \Bigr\} \,.
 \end{eqnarray}
Equivalently,
 \begin{eqnarray}\label{eq:kid5}
\sum \Bigl\{ \kappa_s \,: \, s\in   Q_{\delta-\eps}^{\sharp}(y)\Bigr\} \ge
 \Bigl(1-\delta^{\beta/3}\Bigr)\sum \Bigl\{ \kappa_s \,: \, s\in  S \cap  Q_{\delta+\eps}(y)   \Bigr\} \,.
 \end{eqnarray}

Combining this inequality with \eqref{eq4:8} and \eqref{compa} proves \eqref{eq4:9}.

\medskip
Recall from \eqref{eq:2} that $u=0$ on $\partial B(0,\rho)$.
By considering separately the two cases $A_+>1/2$ and $A_+ \le 1/2$, we infer from \eqref{eq4:51} that
 $$ \forall x \in \Lambda_{\delta}(y) \cap \partial B(0,1+\delta), \qquad u(x) \ge  \tilde{A}:=(1-C\delta^\beta)A \,.
 $$
  Thus, Lemma \ref{l:holder} implies that
\begin{eqnarray}\label{eq4:10}
E\left(u,\{x\in  \Lambda_{\delta}(y): ~ |x|\ge 1+\delta \} \right) \ge
\tilde{A}^d  E(U^{\delta}_{\rho},\Lambda_{\delta}(y))
\ge (1-C\delta)\tilde{A}^d  E(U_{\rho},  \Lambda_{\delta}(y)) \,.
\end{eqnarray}
Here $U^{\delta}_{\rho}$ denotes the radial $d$-harmonic potential of $\bar B(0,1+\delta)$ in $B(0,\rho)$.
 Since $\delta^\beta=o(\delta^{\beta/3})$ and
$\delta=o(\delta^{\beta/3})$, we may decrease $\delta_0$ so that the product
of the two factors in \eqref{eq4:10} is at least
$1-\delta^{\beta/3}$.  After also taking $\eps$ small enough that
$E(U_\rho^\eps,\Lambda_\delta(y))=(1+o(1))E(U_\rho,\Lambda_\delta(y))$, we obtain
\begin{eqnarray}\label{eq4:11}
E\left(u, \{x\in  \Lambda_{\delta}(y): ~ |x|\ge 1+\delta \}\right)
\ge (1-\delta^{\beta/3}) A^d E(U^{\eps}_{\rho}, \Lambda_{\delta}(y)) \,.
\end{eqnarray}
 Combining \eqref{eq4:9} and \eqref{eq4:11}, and then using
\eqref{decomp-ener2.5}, proves \eqref{eq4:2} with the stated factor
$1-3\delta^{\beta/3}$ after taking $\eps$ sufficiently small relative to
$\delta$.

\end{proof}

\section{Asymptotics for $u^\eps$ in the bulk: Proof of Theorem \ref{t:1} } \label{proofthm1}

In this section, we use the lower bound on the energy of $u$ in cones from
Lemma \ref{l:4}, together with global energy minimization, to deduce an upper
bound for the energy of $u$ in every cone.  This will imply that $u$ is close
to the bulk value $A_*$ on $\partial B(0,1+\delta)$.
\begin{lemma}\label{l:5}
Let $\beta=(d-1)/(2d)$ as in Definition \ref{defbad}.
There exist $\delta_0>0$ and $C>0$, with the same parameter dependencies as
in Lemma \ref{l:4}, such that for all
  $\delta \in (0,\delta_0/2)$, points $z\in \mathbb{S}^{d-1}$, and fixed $m>1$, we have
\begin{eqnarray}\label{eq:lem5}
E(u,\Lambda_{\delta}(z))\le  E(\widehat u^\eps_{A_*},\Lambda_{\delta}(z)) +C{\delta^{m\beta/3}} \,,
\end{eqnarray}
provided that $\eps<\eps_0(\delta,m)$ is sufficiently small; the constant
$C$ is independent of $z$. Consequently, there exists $\theta=\theta(d)>0$
such that for $\delta \in (0,\delta_0/2)$ and $\eps$ sufficiently small,
\begin{eqnarray}\label{eq:lem5const}
 |u(x)-A_*| \le C_{\#} \delta^{\theta}   \; \;  \mbox{\rm for all} \; \; x\in B(0,\rho) \;\; \mbox{\rm such that} \; \; |x|=1\pm \delta \,,
\end{eqnarray}
where $C_{\#}$ may depend on $d,\sigma,m_0,\rho,$ and $\tau$, but not on
$\eps,\delta,$ or $z$.
\end{lemma}
As in the previous section, the constants $C_i$ in the proof depend only on
the parameters listed in Lemma \ref{l:4}.

\begin{proof}
Let $\Omega=B(0,\rho) \setminus \Gamma$. Observe that for $x \in  B(0,\rho) $  and  $y \in {\mathbb S}^{d-1}$, we have
 $$x \in \Lambda_{2 \delta}(y)  	\Leftrightarrow y \in Q_{2\delta}(x/|x|) \,.
 $$
 Therefore, by Fubini,
\begin{eqnarray}\label{fub1}
\int_{{\mathbb S}^{d-1}} E(u, \Lambda_{2 \delta}(y)) \, d\mu(y)
&=& \int_{{\mathbb S}^{d-1}} \int_{\Omega \cap \Lambda_{2 \delta}(y)} |\nabla u(x)|^d \, dx \, d\mu(y) \nonumber\\
&=& \int_{\Omega} \left( \int_{Q_{2\delta}(x/|x|)} |\nabla u(x)|^d \, d\mu(y) \right) dx \nonumber\\
&=& \mu(Q_{2\delta})  E(u,B(0,\rho)) \,.
\end{eqnarray}
Define
\begin{eqnarray}\label{eq:bady}
 Y_\delta=\Bigl\{y \in {\mathbb S}^{d-1} \,:\, E(u, \Lambda_{2 \delta}(y))>\delta^{-\beta/2} \delta^{d-1} \Bigr\} \,.
\end{eqnarray}
Since $E(u,B(0,\rho)) \le  C$ and $\mu(Q_{2\delta}) =O(\delta^{d-1})$,  equation \eqref{fub1} implies that
\begin{eqnarray}\label{eq:bady2}
 \mu(Y_\delta)  \le C \delta^{\beta/2} \,.
\end{eqnarray}
Denote $\chi_m(y)=E(\widehat u^\eps_{A_*}, \Lambda_{\delta^m}(y))$ and observe that if $\delta^m<\delta_0$, then Lemma \ref{l:4} and \eqref{eq:bady3}
(applied to $\delta^m$ in place of $\delta$)
imply that for sufficiently small $\eps$,
\begin{eqnarray}\label{eq:bady3.5}
E(u, \Lambda_{\delta^m}(y)) \ge (1-C\delta^{m\beta/3}) \chi_m(y)
\quad\mbox{for }y \in {\mathbb S}^{d-1} \setminus Y_{\delta^m} \,.
\end{eqnarray}
Indeed, the direct product of the two estimates is
$(1-3\delta^{m\beta/3})(1-\delta^m)^d$, which is bounded below by the
displayed factor after decreasing $\delta_0$.

Next, note that for  $z \in {\mathbb S}^{d-1}$ and   $y \in {\mathbb S}^{d-1} \setminus Q_{\delta+\delta^m}(z)$, we have
$$ x \in \Lambda_{\delta^m}(y) \, \Rightarrow \, \Bigl\{x \in  B(0,\rho) \setminus \Lambda_\delta(z) \; \; {\rm and} \;\; y \in Q_{\delta^m}(x/|x|) \Bigr\} \,.
$$
Therefore, by Fubini
\begin{eqnarray}\label{fub2}
\!\!\! \! \underset{{\mathbb S}^{d-1} \setminus Q_{\delta+\delta^m}(z)}  \int \!\! \! \! \! \!\!\! \! E(u, \Lambda_{  \delta^m}(y)) \, d\mu(y)
 = \!\!\! \! \! \underset{{\mathbb S}^{d-1} \setminus Q_{\delta+\delta^m}(z)} \int \!\! \!  \!\int_{\Omega \cap \Lambda_{  \delta^m}(y)} \!\!
  |\nabla u(x)|^d \, dx \, d\mu(y)
  \le \!\!\! \! \!\!\! \! \underset{\Omega \setminus \Lambda_\delta(z)} \int \!\!\! \! \int_{Q_{\delta^m}(x/|x|)} \! \! \! |\nabla u(x)|^d \, d\mu(y) \, dx   \,.
\end{eqnarray}
The right-hand side equals
$\mu(Q_{\delta^m})E(u,\Omega\setminus\Lambda_\delta(z))$; hence
\eqref{fub2} gives
\begin{eqnarray}\label{fub2.5}
\mu(Q_{\delta^m}) E(u, \Omega \setminus \Lambda_\delta(z)) \ge \!\! \underset{{\mathbb S}^{d-1} \setminus Q_{\delta+\delta^m}(z)}   \int E(u, \Lambda_{  \delta^m}(y)) \, d\mu(y)
 \ge \!\! \underset{{\mathbb S}^{d-1} \setminus [Q_{\delta+\delta^m}(z) \cup   Y_{\delta^m}]} \int (1-C\delta^{m\beta/3}) \chi_m(y) \, d\mu(y) \,,
\end{eqnarray}
where we have used \eqref{eq:bady3.5} in the last step. By Fubini,
\begin{eqnarray}\label{fub3}
\int_{{\mathbb S}^{d-1}}  \chi_m(y) \, d\mu(y) =\mu(Q_{\delta^m})  E(\widehat u^\eps_{A_*},B(0,\rho)) \,
\end{eqnarray}
and
\begin{eqnarray}\label{fub4}
\int_{ Q_{\delta-\delta^m}(z) }  \chi_m(y) \, d\mu(y)  \le \mu(Q_{\delta^m})  E(\widehat u^\eps_{A_*}, \Lambda_\delta(z)) \,.
\end{eqnarray}
Write $Y^+=Y_{\delta^m} \cup \Bigl(Q_{\delta+\delta^m}(z) \setminus Q_{\delta-\delta^m}(z)\Bigr)$, so that for small $\delta$,
\begin{eqnarray}\label{fub5}
\mu(Y^+) \le C\delta^{m\beta/2}+C\delta^{d-2+m}\le C\delta^{m\beta/2} \,.
\end{eqnarray}
If $\tau<\infty$, \eqref{c13b} implies the bound
$\chi_m(y) \le C\mu(Q_{\delta^m})$ for $\eps$ small enough. If
$\tau=\infty$, then $A_*=1$ and the same bound follows directly from the
radial bulk energy. Together with \eqref{fub5}, this gives
\begin{eqnarray}\label{fub6}
\int_{ Y^+ }  \chi_m(y) \, d\mu(y) \le C \mu(Q_{\delta^m}) \delta^{m\beta/2}  \,.
\end{eqnarray}
Subtracting \eqref{fub4} and \eqref{fub6} from \eqref{fub3} yields
\begin{eqnarray}\label{fub7}
\underset{{\mathbb S}^{d-1} \setminus [Q_{\delta+\delta^m}(z) \cup   Y_{\delta^m}]} \int   \chi_m(y) \, d\mu(y)
\ge   \mu(Q_{\delta^m}) \Bigl[  E(\widehat u^\eps_{A_*}, B(0,\rho) \setminus \Lambda_\delta(z)) -C\delta^{m\beta/2} \Bigr] \,,
\end{eqnarray}
whence by \eqref{fub2.5},
\begin{eqnarray}\label{fub8}
  E(u, \Omega  \setminus \Lambda_\delta(z)) \ge
   (1-C\delta^{m\beta/3})  \Bigl[  E(\widehat u^\eps_{A_*}, B(0,\rho) \setminus \Lambda_\delta(z)) -C\delta^{m\beta/2} \Bigr] \,.
\end{eqnarray}

Thus,
\begin{eqnarray}\label{fub8.5} E(\widehat u^\eps_{A_*}) \ge  E(u) &=& E(u,  \Lambda_\delta(z)) +E(u, \Omega \setminus \Lambda_\delta(z)) \\ \nonumber
&\ge& E(u,  \Lambda_\delta(z))+
 (1-C\delta^{m\beta/3})  \Bigl[  E(\widehat u^\eps_{A_*}, \Omega \setminus \Lambda_\delta(z)) - C\delta^{m\beta/2} \Bigr] \,.
\end{eqnarray}
Rearranging terms and using the uniform bound $E(\widehat u^\eps_{A_*})\le C$, we conclude that
 \begin{eqnarray}\label{fub9}
  E(u,  \Lambda_\delta(z)) \le  E(\widehat u^\eps_{A_*},  \Lambda_\delta(z)) + C\delta^{m\beta/3}\,,
\end{eqnarray}
provided $\delta<\delta_0$ and $\eps<\eps_0(\delta,m)$ are sufficiently
small. Applying this inequality with $2\delta$ in place of $\delta$, the hypothesis of
 Lemma \ref{l:4} is satisfied,
  provided $m$ is chosen to satisfy  $m\beta/3>d$.   Therefore, by \eqref{eq4:5},  we have    that for  sufficiently small $\eps$,
  \begin{eqnarray}\label{fub9.5}
  |u(x)-A| \le C \delta^\beta  \;\; \mbox{ \rm for all} \;\; x \in \Lambda_\delta(z) \;\; \mbox{\rm such that} \;\; |x|=1\pm \delta \,,
  \end{eqnarray}
 where $A=A(z,\delta,\eps)$. By \eqref{eq4:2} and \eqref{fub9},
 \begin{eqnarray}\label{fub9.7}
    (1- 3\delta^{\beta/3}) E(\widehat u^\eps_{A},\Lambda_{\delta}(z))\, \le \, E(\widehat u^\eps_{A_*},  \Lambda_\delta(z)) + C\delta^{d}\,,
\end{eqnarray}
 since  $m\beta/3>d$. Now we separate cases.

 \medskip

 \noindent{\bf Case 1.} If $\tau \in (0,\infty)$, then \eqref{fub9.7} and Corollary \ref{cor:en} yield, for sufficiently  small $\eps>0$, that
  \begin{eqnarray}\label{fub10}
   (1-3 \delta^{\beta/3}) \mu(Q_\delta) \varphi_\tau(A)   \le \mu(Q_\delta) \varphi_\tau(A_*)+ C\delta^{d}\,.
\end{eqnarray}
Thus, since $\varphi_\tau(A_*) \le C$, we infer that
$$C\delta^{\beta/3}  \ge \varphi_\tau(A)-\varphi_\tau(A_*) \ge c_{\#}|A-A_*|^{2} \,,
$$
where $c_{\#}=c_{\#}(d,\tau,\sigma,\rho)>0$.  Therefore,
$$|A-A_*|\le C\delta^{\beta/6} \,.
$$
 In conjunction with \eqref{fub9.5}, this yields the final claim of the lemma.

 \medskip

 \noindent{\bf Case 2.} If $\tau =0$, then $A_*=0$, so $E(\widehat u^\eps_{A_*}) \to 0$ as $\eps \to 0$ by \eqref{limener}. Since
   $$   E(\widehat u^\eps_{A},  \Lambda_\delta(z))\ge CA^d \delta^{d-1} $$
by \eqref{eqh:6}, we infer from \eqref{fub9.7} that
  $|A-A_*|=A =O(\delta^{1/d})$ for small $\eps$.

 \medskip

  \noindent{\bf Case 3.} If $\tau =\infty$, then $A_*=1$ and
  \begin{eqnarray}\label{fubnew}
    E(\widehat u^\eps_{A_*},  \Lambda_\delta(z)) = E( U^{\eps}_{\rho},  \Lambda_\delta(z)) \le C \delta^{d-1}
\end{eqnarray}
by \eqref{eqh:6}.
  On the other hand, for $A<1$, the definition of $\widehat u_A^\eps$ and \eqref{kappa-uniform} imply that, for $\eps>0$ small enough,
 \begin{eqnarray}\label{fubnew1} \frac{E(\widehat u^\eps_{A},\Lambda_{\delta}(z))}{(1-A)^d} \ge  \sum_{s \in S \cap Q_{\delta-\eps}(z)} \kappa_s \ge
  C (\delta/\eps)^{d-1} \frac{1}{(\log(1/(10\alpha)))^{d-1}}
  = C\delta^{d-1}\left(\eps\log(1/(10\alpha))\right)^{1-d}.
  \end{eqnarray}
  Since $\tau=\infty$ means $\eps\log(1/\alpha)\to0$, the right-hand side of \eqref{fubnew1} tends to $\infty$ as $\eps \downarrow 0$.  Thus, by \eqref{fub9.7}, we have
  $|A-A_*|=1-A \le \delta$ for small enough $\eps$.
\end{proof}

\begin{proof}[Proof of Theorem \ref{t:1}]
  By Lemma \ref{l:5}, if $\delta<\delta_0(d,\sigma,m_0,\rho)$, then for sufficiently small $\eps$, we have
  $$ \forall x \in \partial B(0,1+\delta) \cup \partial B(0,1-\delta), \qquad  A_*-C_{\#} \delta^\theta  \le u(x) \le A_*+C_{\#} \delta^\theta  \,,
  $$
  where $\theta, C_{\#}$ do not depend on $\eps,\delta$.
The comparison principle (Lemma  \ref{l:comp}) then implies that for sufficiently small $\eps$,
\begin{eqnarray}\label{bulk1}
 \forall x \in B(0,\rho) \setminus B(0,1+\delta), \qquad  (A_*-C_{\#} \delta^\theta) U^{\delta}_{\rho}(x) \le u(x) \le (A_*+C_{\#} \delta^\theta)U^{\delta}_{\rho}(x)  \,.
 \end{eqnarray}
 and
\begin{eqnarray}\label{bulk2}
\forall x\in B(0,1-\delta),\qquad
A_*-C_{\#}\delta^\theta\le u(x)\le A_*+C_{\#}\delta^\theta.
\end{eqnarray}
For the latter conclusion, the maximum principle suffices. Letting first
$\eps\downarrow0$ for fixed $\delta$ and then $\delta\downarrow0$ proves
uniform convergence on compact subsets of
$\bar B(0,\rho)\setminus\mathbb S^{d-1}$. Compact sets meeting
$\partial B(0,\rho)$ are covered by \eqref{bulk1}, and both $u^\eps$ and
$A_*U_\rho$ vanish on that boundary.

It remains to prove the capacity assertion.  Fix $\delta>0$ sufficiently
small and choose $m$ so that $m\beta/3>d$.  Lemma \ref{l:5}, applied with
$2\delta$, and Corollary \ref{cor:en} give, uniformly in
$y\in\mathbb S^{d-1}$,
\begin{eqnarray}
E(u,\Lambda_{2\delta}(y))
\le E(\widehat u^\eps_{A_*},\Lambda_{2\delta}(y))+C\delta^d
\le C\delta^{d-1}.
\end{eqnarray}
After decreasing $\delta$, this is at most
$\delta^{d-1-\beta/2}$, so Lemma \ref{l:4} applies in every cone.  Combining
it with \eqref{eq:bady3} yields
\begin{eqnarray}
E(u,\Lambda_\delta(y))
\ge(1-C\delta^{\beta/3})
E(\widehat u^\eps_{A_*},\Lambda_\delta(y)).
\end{eqnarray}
Averaging over $y$ and using Fubini gives
\begin{eqnarray}\label{eq:global-energy-squeeze}
(1-C\delta^{\beta/3})E(\widehat u^\eps_{A_*})
\le E(u)\le E(\widehat u^\eps_{A_*}),
\end{eqnarray}
where the upper bound follows from Lemma \ref{lem:variational}. Set
\[
L_\tau:=\frac{\omega_{d-1}}{(\log\rho)^{d-1}}A_*^{d-1}.
\]
If $\tau<\infty$, equations \eqref{limener} and \eqref{phimin} give
\begin{eqnarray}\label{eq:finite-capacity-energy}
E(\widehat u^\eps_{A_*})\longrightarrow
\varphi_\tau(A_*)=L_\tau.
\end{eqnarray}
If $\tau=\infty$, then $A_*=1$ and directly
\begin{eqnarray}\label{eq:infinite-capacity-energy}
E(\widehat u^\eps_{A_*})=E(U_\rho^\eps)\longrightarrow
\frac{\omega_{d-1}}{(\log\rho)^{d-1}}=L_\tau.
\end{eqnarray}
For each fixed sufficiently small $\delta$,
\eqref{eq:global-energy-squeeze} therefore implies
\[
(1-C\delta^{\beta/3})L_\tau
\le\liminf_{\eps\downarrow0}E(u^\eps)
\le\limsup_{\eps\downarrow0}E(u^\eps)
\le L_\tau.
\]
Letting $\delta\downarrow0$ proves $E(u^\eps)\to L_\tau$. Since
$E(u^\eps)={\rm cap}_d(\Gamma_\eps,B(0,\rho))$, this proves
\eqref{eq:thm1B}.
\end{proof}

\section{Separation theorem} \label{necksec}

In this section we formulate an analogue of \cite[Theorem~5.1]{GNP} for the case $p=d$.

\begin{theorem} \label{neckthm}
Assume $p=d$.
Let $0<\eps<\delta/10<1/10$, and let $\emptyset \ne S \subset \partial B(0,1)$ be a set of anchors, with Euclidean distance at least $\eps$ between any two anchors and $\alpha <\min\{1/80,\exp(-(\tau_1 \eps)^{-1}) \} \,$ for some  $\tau_1 \in (0,\infty).$
For $\zeta>0$, write    $\Omega_\zeta = B(0,1 +\delta) \setminus \bigl[ \cup_{s \in S} \bar{B}(s, \zeta \eps) \bigr]$.
Let $w: \bar{B}(0,1+\delta) \to [0,1]$ be the Perron solution of the  boundary value problem
\begin{eqnarray}
\left\{
\begin{array}{lll}
\Delta_p w = 0 & \mbox{in} &\Omega_{\alpha}, \\
w=0 & \mbox{on} & \partial B(0,1+\delta), \\
w=1 & \mbox{on}   &  \cup_{s\in S} \partial B(s,\alpha \eps) \,.
\end{array}
\right.
\end{eqnarray}
Then for some $C=C(d)$, we have
\begin{equation} \label {eq:neckth}
\sup_{z \in \Omega_{1/10}} |w(z)| \le C\tau_1 \delta \,.
\end{equation}
\end{theorem}

\begin{figure}[!ht]%
	\centering \includegraphics[width=4.in]{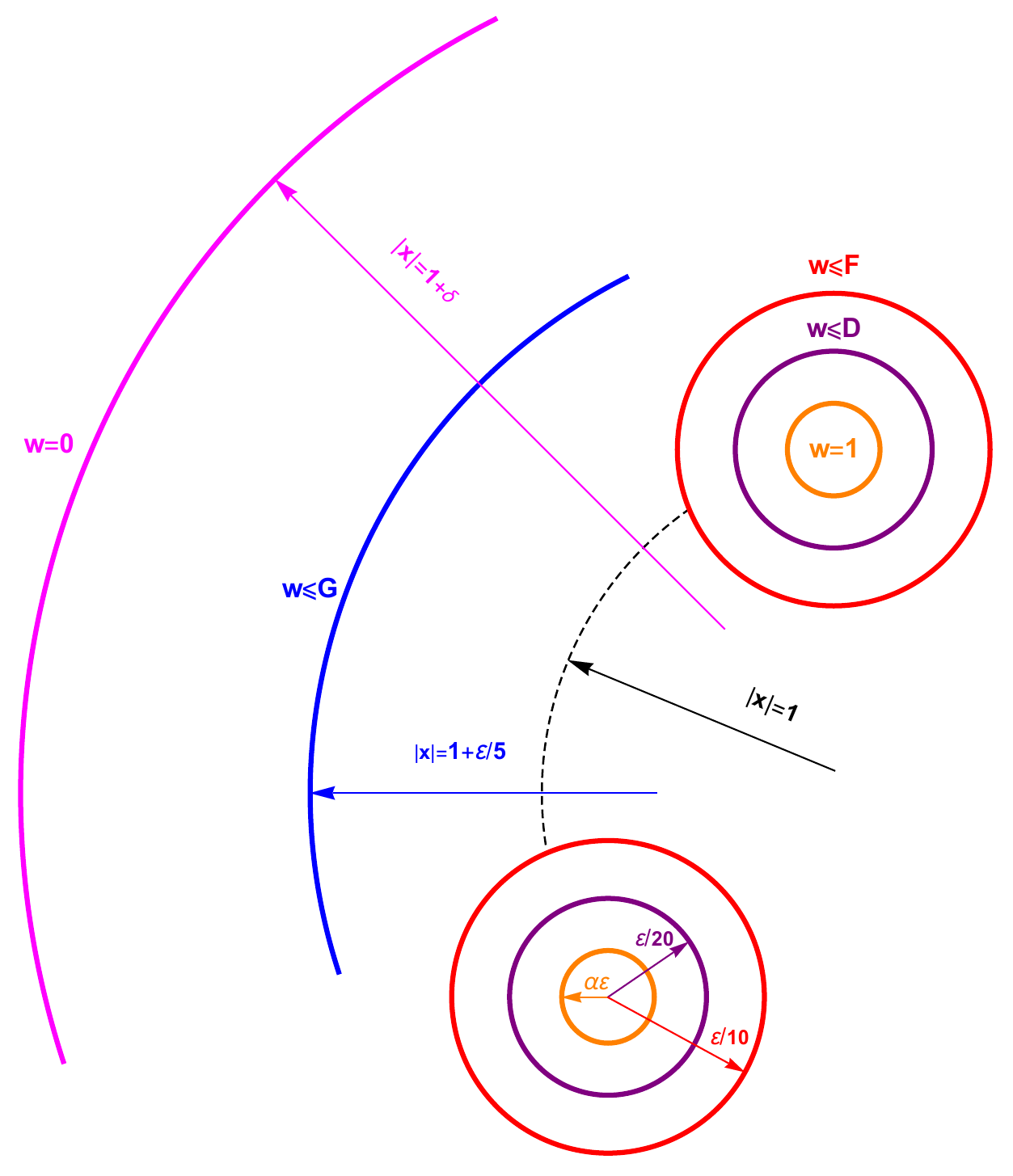}
	\caption{The inner balls are centered at anchors on the unit sphere and have radius $\alpha \eps$; on their boundary, $w=1$. Each is surrounded by two concentric spheres of radii $\eps/20$ and $\eps/10$. These small spheres lie inside the larger balls of radii $1+\eps/5$ and $1+\delta$ centered at the origin.}
	\label{fig:neck}
\end{figure}

The proof of Theorem \ref{neckthm} is based on the maximum principle and the following lemma.

\begin{lemma}\label{necklem}
Let \begin{align}
D & := \sup\Bigl\{ w(x) \mid   x \in \cup_{s \in S} \; \partial B\left (s, \frac{\eps}{20}\right) \Bigr\}\, ,\notag \\
F & := \sup\Bigl\{ w(x) \mid  x \in \cup_{s \in S} \; \partial B\left(s, \frac{\eps}{10}\right) \Bigr\} \, , \notag \\
G & := \sup\Bigl\{ w(x) \mid x \in \partial B\left(0 , 1 + \frac{\eps}{5}\right)\Bigl\} \,. \notag
\end{align}
Then there exist constants $c_1, c_2>0$ and $c_3 \in (0,1)$ that only depend on $d$,
such that
\begin{align}
(a) \qquad D & \leq F  + c_1 \tau_1  \eps(1- F) \,, \label{D_ineq} \\
(b) \qquad G & \leq (1 - c_2 \eps/\delta)F \label{G_ineq} \,,\\
(c) \qquad F & \leq (1-c_3)D + c_3 G \label{F_ineq} \,.
\end{align}
\end{lemma}
\begin{proof}
Let $0<r<R$ and set
\begin{equation} \label{funda}
h_{r,R}(x)=\frac{\log(|x|/R)}{\log(r/R)}  \; \mbox{ \rm for }\;  r < |x| < R\,.
\end{equation}
Observe that $h_{r,R}$ is positive and $p$-harmonic in $B(0,R)\setminus \bar  B(0,r)$ and takes values $0$ and $1$ on $\partial B(0,R)$ and $\partial B(0,r)$ respectively.

\medskip

(a) ~Given the definition of $F,$  the comparison principle (Lemma \ref{l:comp}) implies that for each $s\in S$, we have
\begin{eqnarray}\label{neckeq1}
w(x)\le F+(1-F)\nu_1(x-s) , \qquad \forall x\in  B\left(s,\frac{\eps}{10}\right) \setminus  B(s,\alpha\eps),
\end{eqnarray}
where $\nu_1 (x)=h_{r,R}(x)$ with $r=\alpha \eps$,~$R={\eps}/{10}.$
Hence, for $|x|=\eps/20$,
\begin{eqnarray} \label{neckeq2}
\nu_1(x)=\frac{\log 2}{\log(1/(10\alpha))}.
\end{eqnarray}
Put $q=\tau_1\eps$. If $q\le(2\log10)^{-1}$, the assumption
$\alpha<\exp[-1/q]$ gives
$\log(1/(10\alpha))\ge(2q)^{-1}$. If $q>(2\log10)^{-1}$, the additional
bound $\alpha<1/80$ gives
$\nu_1\le\log2/\log8\le Cq$. Thus in all cases
\begin{eqnarray}
\nu_1(x)\le c_1\tau_1\eps \qquad \mbox{for } |x|=\eps/20.
\end{eqnarray}
We infer from \eqref{neckeq1} that
\begin{eqnarray}
  w(x) \leq F + (1-F) c_1\tau_1  \eps    \qquad \forall x\in \partial B\left(s,\frac{\eps}{20}\right) \,.
\end{eqnarray}
This gives \eqref{D_ineq}.

\medskip

(b) First, the maximum principle gives
\begin{equation} \label{eq:boundviaF}
\sup_{z \in \Omega_{1/10}} w(z) \le F \,.
\end{equation}
Therefore, $w \le F$ on $\partial B(0,1+\eps/10)$. Hence, by the comparison principle,
\begin{eqnarray}\label{neckeq3}
w(x)\le F \nu_2 (x) , \qquad \forall x\in B\left(0, 1+\delta \right) \setminus B\left (0,1+\frac{\eps}{10}\right)\,,
\end{eqnarray}
where $\nu_2(x)=h_{r,R}(x)$ with $r=1+\eps/10$ and $R=1+\delta,$ that is
\begin{eqnarray}
\nu_2(x)=\frac{\log\left(\frac{|x|}{1+\delta}\right)}{\log\left(\frac{1+\eps/10}{1+\delta}\right)}.
\end{eqnarray}
Hence,
\begin{eqnarray}
1 -  \nu_2\left(1+\frac{\eps}{5}\right)=\frac{\log\left(\frac{1+\eps/5}{1+\eps/10}\right)}{\log\left(\frac{1+\delta}{1+\eps/10}\right)}.
\end{eqnarray}
Because $0<\eps<\delta/10<1/10$, the elementary bounds for
$\log(1+t)$ give
\begin{eqnarray}
\log\left(\frac{1+\eps/5}{1+\eps/10}\right)\ge c\eps,
\qquad
\log\left(\frac{1+\delta}{1+\eps/10}\right)\le\log(1+\delta)\le\delta.
\end{eqnarray}
Consequently,
  \begin{equation} \label{vlocal}
  1 -  \nu_2\left(1+\frac{\eps}{5}\right)  \geq c_2 \frac{\eps}{\delta}\,.
\end{equation}
  Combining \eqref{neckeq3} and  \eqref{vlocal}, we have that
\begin{eqnarray}
w(x) \le \left ( 1-c_2\frac{\eps}{\delta}\right) F, \qquad  \forall x\in\partial B\left(0,1+\frac{\eps}{5}\right).
\end{eqnarray}
  This gives \eqref{G_ineq}.

\medskip

(c) We include the short argument from \cite[Theorem~5.1]{GNP}.  Let
$\xi=(1,0,\ldots,0)$ and consider
\begin{eqnarray}
H=\{x\in B(0,8)\setminus\bar B(4\xi,1):x_1>0\}.
\end{eqnarray}
Define $f\in C(\partial H)$ as follows: set $f=1$ on the two spherical
parts of $\partial H$, and on the flat part set
\begin{eqnarray}
f(x)=\max\{0,|x|-7\}.
\end{eqnarray}
The definitions agree at the rim $|x|=8$. Since $H$ is bounded and
Lipschitz, the continuous boundary data are resolutive. Let
$\psi\in C(\bar H)$ be the $d$-harmonic Perron solution in $H$ with boundary
values $f$. Then $0\le\psi\le1$, and $\psi$ is not constant. The strong
maximum principle and continuity imply
that, for some $c_3=c_3(d)\in(0,1)$,
\begin{eqnarray}\label{psimax}
\max_{x\in\partial B(4\xi,2)}\psi(x)=1-c_3.
\end{eqnarray}
For $s\in S$, put $s_*=(1+\eps/5)s$ and
\begin{eqnarray}
H_s=\{x\in B(s_*,2\eps/5)\setminus\bar B(s,\eps/20):
\langle x,s\rangle<1+\eps/5\}.
\end{eqnarray}
This is a rotated, translated, and $\eps/20$-scaled copy of $H$.  Anchor
separation ensures that no other removed ball meets $H_s$.
Let $\psi_s$ denote the corresponding rotated, translated, and
$\eps/20$-scaled copy of $\psi$ on $H_s$.
The maximum principle in $\Omega_{1/20}$ first gives $w\le D$ there, and
hence $D\ge G$.  The radial annulus
$\{1+\eps/5\le |x|\le1+\delta\}$ is free of removed balls, so the maximum
principle also gives $w\le G$ throughout that annulus. Thus $w\le G$ on
the flat boundary portion of $H_s$ and $w\le D$ on both curved portions.
On the flat portion $\psi_s\ge0$, and on the curved portions $\psi_s=1$.
Hence comparison with $G+(D-G)\psi_s$, followed by \eqref{psimax} on
$\partial B(s,\eps/10)$, gives
\begin{eqnarray}
F\le G+(1-c_3)(D-G)=(1-c_3)D+c_3G,
\end{eqnarray}
which proves (c).
\end{proof}

\begin{proof}[Proof of Theorem \ref{neckthm}]
Combining \eqref{D_ineq}-\eqref{F_ineq}, we have
\begin{eqnarray}
& F  \leq (1-c_3) \bigl[ F + c_1 \tau_1 \eps (1-F)\bigr] + c_3 \bigl[1 - c_2 \eps/\delta \bigr]F.
\end{eqnarray}
 This  gives
 \begin{eqnarray}
 0  \leq (1-c_3)c_1\tau_1 \eps(1-F) - c_3 c_2(\eps/\delta) F\,,
\end{eqnarray}
which in turn implies that
\begin{eqnarray}
 c_3 c_2 (\eps/\delta)F \leq   c_1 \tau_1 \eps\,,
 \end{eqnarray}
 whence
 \begin{eqnarray}
& F \leq \left(\frac{c_1 }{c_3 c_2} \right) \tau_1 \delta \,.
\end{eqnarray}
Thus there exists a constant $c_{4}>0$ such that
$F \leq c_{4}\tau_1 \delta$. By \eqref{eq:boundviaF}, this completes the proof.
\end{proof}

\section{Asymptotics for $u^\eps$ near the unit sphere: Proof of Theorem \ref{main2} } \label{proofmain}

Since the critical case has $p=d\ge2$, we may use Clarkson's inequality, in the slightly more general form given in \cite{Boas} for functions $f,g$ taking values in ${\mathbb R}^d$:
\begin{eqnarray} \label{clark1}
 p \ge 2 \Rightarrow \quad \left\| \frac{f + g}{2} \right\|_{ p}^p + \left\| \frac{f - g}{2} \right\|_{p}^p \le
 \frac{1}{2} \left( \| f \|_{p}^p + \| g \|_{ p}^p \right) \,,
\end{eqnarray}

\begin{proof}[Proof of Theorem \ref{main2}]
Recall that $\beta=(d-1)/(2d)$.  From \eqref{eq:global-energy-squeeze}, for
each fixed sufficiently small $\delta$ and all sufficiently small $\eps$,
\begin{eqnarray}\label{soclose}
0\le E(\widehat u^\eps_{A_*})-E(u^\eps)\le C\delta^{\beta/3}.
\end{eqnarray}
The functions $u^\eps$, $\widehat u^\eps_{A_*}$, and their midpoint belong
to the same convex admissible class.  By Lemma \ref{lem:variational}, the
energy of the midpoint is at least $E(u^\eps)$.  Clarkson's inequality gives
\begin{eqnarray}
2^{-d}\|\nabla u^\eps-\nabla\widehat u^\eps_{A_*}\|_d^d
\le\frac12\bigl(E(\widehat u^\eps_{A_*})-E(u^\eps)\bigr)
\le C\delta^{\beta/3}.
\end{eqnarray}
Consequently, for every fixed sufficiently small $\delta$,
\begin{eqnarray}
\limsup_{\eps\to0}
\|\nabla u^\eps-\nabla\widehat u^\eps_{A_*}
\|_{L^d(B(0,\rho))}^d\le C\delta^{\beta/3}.
\end{eqnarray}
Letting $\delta\to0$ proves convergence of the gradients. Since
$u^\eps-\widehat u^\eps_{A_*}\in H^{1,d}_0(B(0,\rho))$, Poincar\'e's
inequality also gives convergence in $L^d$.

We next prove uniform convergence to the shape-adapted ansatz.  The bulk
estimate \eqref{bulk1}, together with
\begin{eqnarray}\label{eq:bulk-potentials-close}
\|U_\rho^\delta-U_\rho^\eps\|_{L^\infty(B(0,\rho)\setminus B(0,1+\delta))}
\le C\delta,
\end{eqnarray}
already gives the desired bound outside $B(0,1+\delta)$.

If $\tau=\infty$, then $A_*=1$ and both ansatz functions equal
$U_\rho^\eps$. Lemma \ref{l:5} gives
$u^\eps\ge1-C\delta^\theta$ on $\partial B(0,1+\delta)$, while
$u^\eps=1\ge1-C\delta^\theta$ $d$-quasieverywhere on $\Gamma$. The weak
comparison principle from Lemma \ref{lem:variational}, applied in
$B(0,1+\delta)\setminus\Gamma$, yields the same lower bound
$d$-quasieverywhere there and hence, by continuity off $\Gamma$, pointwise
there. Since $u^\eps\le1$ and $|1-U_\rho^\eps|\le C\delta$ in
$B(0,1+\delta)$, this completes the endpoint case.

Assume henceforth that $\tau<\infty$.  Replacing $\theta$ by
$\min\{\theta,1\}$ if necessary, let $w$ be the function from Theorem
\ref{neckthm}, with a fixed $\tau_1>\tau$ when $\tau>0$ and any fixed
$\tau_1>0$ when $\tau=0$.  Hypothesis $  {\bf  H_3}$ ensures that the required
smallness condition on $\alpha$ holds for all sufficiently small $\eps$.
The theorem gives $w\le C(d)\tau_1\delta$ on
\begin{eqnarray}
\Omega_{1/10}=B(0,1+\delta)\setminus
\bigcup_{s\in S}\bar B(s,\eps/10).
\end{eqnarray}
On the boundary of $\Omega_\alpha$, we have
$u^\eps-A_*-C\delta^\theta\le w$: this follows from Lemma \ref{l:5} on
$\partial B(0,1+\delta)$ and from $0\le u^\eps\le1=w$ on the inner spheres.
Comparison therefore gives
$u^\eps\le A_*+C\delta^\theta+C(d)\tau_1\delta$ on
$\partial B(s,\eps/10)$.  Since $\theta\le1$ and $\tau_1$ is fixed, the
last term can be absorbed into the constant.

For the lower bound, put $k_\delta=A_*-C\delta^\theta$. Lemma \ref{l:5}
gives $u^\eps\ge k_\delta$ on $\partial B(0,1+\delta)$, while
$u^\eps=1\ge k_\delta$ $d$-quasieverywhere on $\Gamma$. The weak comparison
principle from Lemma \ref{lem:variational}, applied in
$B(0,1+\delta)\setminus\Gamma$, gives $u^\eps\ge k_\delta$
$d$-quasieverywhere there. Continuity away from $\Gamma$ makes this estimate
pointwise on the outer small spheres. Thus, for every $s\in S$,
\begin{eqnarray}\label{eq:outer-small-ball-bound}
A_*-C\delta^\theta\le u^\eps\le A_*+C\delta^\theta
\quad\hbox{on }\partial B(s,\eps/10).
\end{eqnarray}
On $\Omega_{1/10}$ the bump terms in
$\widehat u^\eps_{A_*}$ vanish, and
$|A_*U_\rho^\eps-A_*|\le C\delta$. Consequently,
\begin{eqnarray}\label{so4}
|u^\eps-\widehat u^\eps_{A_*}|\le C\delta^\theta
\quad\hbox{in }\Omega_{1/10}.
\end{eqnarray}

Finally fix $s\in S$ and put
$D_s=B(s,\eps/10)\setminus(s+\alpha\eps K_s)$. If
$x\in B(s,\eps/10)$, then $|x|<1+\eps$, so $U_\rho^\eps(x)=1$.
Moreover, the balls $B(t,\eps/10)$ are pairwise disjoint; hence all bump
terms indexed by $t\ne s$ vanish at $x$. Therefore, in $D_s$,
\begin{eqnarray}
\widehat u^\eps_{A_*}(x)=A_*+(1-A_*)V^s_{1/(10\alpha)}
\left(\frac{x-s}{\alpha\eps}\right).
\end{eqnarray}
This function is $d$-harmonic, equals $A_*$ on the outer sphere, and equals
$1$ $d$-quasieverywhere on the inner compact set.  The same inner boundary
value holds for $u^\eps$, while \eqref{eq:outer-small-ball-bound} controls
the outer boundary.  Weak comparison, in the form stated in Lemma
\ref{lem:variational}, therefore yields
\begin{eqnarray}
|u^\eps-\widehat u^\eps_{A_*}|\le C\delta^\theta\quad\hbox{in }D_s,
\end{eqnarray}
uniformly in $s$.  Letting $\eps\to0$ and then $\delta\to0$ proves the
$L^\infty$ convergence to $\widehat u^\eps_{A_*}$.  Lemma
\ref{lem:ansatz-comparison} transfers both the Sobolev and uniform
convergences to the explicit   ansatz $u^\eps_{A_*}$.
\end{proof}

\noindent{\bf Acknowledgments.} We thank Fedor Nazarov for valuable discussions. The work of P. V. Gordon was supported in part by  US--Israel BSF grant 2024033. The research of Y. Peres was supported by National Natural Science Foundation of China grant RFIS-W2531011.


\begin{bibdiv}
\begin{biblist}
\bib{Evanspde}{book}{
   author={Evans, Lawrence C.},
   title={Partial differential equations},
   series={Graduate Studies in Mathematics},
   volume={19},
   publisher={American Mathematical Society, Providence, RI},
   date={1998},
   pages={xviii+662},
   isbn={0-8218-0772-2},
   doi={10.1090/gsm/019},
}
\bib{GNP}{article}{
  author={Gordon, Peter V.},
  author={Nazarov, Fedor},
  author={Peres, Yuval},
  title={A basic homogenization problem for the $p$-Laplacian in
    $\mathbb{R}^d$ perforated along a sphere: $L^\infty$ estimates},
  journal={Potential Analysis},
  volume={61},
  number={4},
  date={2024},
  pages={701--729},
  doi={10.1007/s11118-024-10126-8},
}


	

\bib{Boas}{article}{
  author={Boas, Ralph P., Jr.},
  title={Some uniformly convex spaces},
  journal={Bulletin of the American Mathematical Society},
  volume={46},
  number={4},
  date={1940},
  pages={304--311},
  doi={10.1090/S0002-9904-1940-07207-6},
}

\bib{Clarkson}{article}{
  author={Clarkson, James A.},
  title={Uniformly convex spaces},
  journal={Transactions of the American Mathematical Society},
  volume={40},
  number={3},
  date={1936},
  pages={396--414},
  doi={10.1090/S0002-9947-1936-1501880-4},
}

\bib{CM97}{article}{
  author={Cioranescu, Doina},
  author={Murat, Fran\c{c}ois},
  title={A strange term coming from nowhere},
  book={
    title={Topics in the Mathematical Modelling of Composite Materials},
    editor={Cherkaev, Andrej},
    editor={Kohn, Robert},
    series={Progress in Nonlinear Differential Equations and Their Applications},
    volume={31},
    publisher={Birkh\"auser Boston},
    address={Boston, MA},
    date={1997},
  },
  pages={45--93},
  doi={10.1007/978-1-4612-2032-$9_4$},
}

\bib{KS14}{article}{
  author={Karakhanyan, Aram L.},
  author={Str\"omqvist, Martin H.},
  title={Application of uniform distribution to homogenization of a thin
    obstacle problem with $p$-Laplacian},
  journal={Communications in Partial Differential Equations},
  volume={39},
  number={10},
  date={2014},
  pages={1870--1897},
  doi={10.1080/03605302.2014.895013},
}

\bib{KS16}{article}{
  author={Karakhanyan, Aram L.},
  author={Str\"omqvist, Martin},
  title={Estimates for capacity and discrepancy of convex surfaces in
    sieve-like domains with an application to homogenization},
  journal={Calculus of Variations and Partial Differential Equations},
  volume={55},
  number={6},
  date={2016},
  pages={Article no. 138, 14 pp.},
  doi={10.1007/s00526-016-1088-2},
}

\bib{GPPS}{article}{
  author={G\'{o}mez, D.},
  author={P\'{e}rez, E.},
  author={Podolskii, A. V.},
  author={Shaposhnikova, T. A.},
  title={Homogenization of variational inequalities for the $p$-Laplace
    operator in perforated media along manifolds},
  journal={Applied Mathematics \& Optimization},
  volume={79},
  number={3},
  date={2019},
  pages={695--713},
  doi={10.1007/s00245-017-9453-x},
}

\bib{PS15}{article}{
  author={Podol'skii, A. V.},
  author={Shaposhnikova, T. A.},
  title={Homogenization for the $p$-Laplacian in an $n$-dimensional domain
    perforated by very thin cavities with a nonlinear boundary condition on
    their boundary in the case $p=n$},
  journal={Doklady Mathematics},
  volume={92},
  number={1},
  date={2015},
  pages={464--470},
  doi={10.1134/S1064562415040201},
}

\bib{PS19}{article}{
  author={Podolskiy, A. V.},
  author={Shaposhnikova, T. A.},
  title={Homogenization of a boundary value problem for the $n$-Laplace
    operator on an $n$-dimensional domain with rapidly alternating boundary
    condition type: The critical case},
  journal={Differential Equations},
  volume={55},
  number={4},
  date={2019},
  pages={523--531},
  doi={10.1134/S0012266119040104},
}

\bib{MKhomo}{book}{
  author={Marchenko, Vladimir A.},
  author={Khruslov, Evgueni Ya.},
  title={Homogenization of Partial Differential Equations},
  series={Progress in Mathematical Physics},
  volume={46},
  publisher={Birkh\"auser Boston},
  address={Boston, MA},
  date={2006},
  pages={xiv+402},
  doi={10.1007/978-0-8176-4468-0},
}

\bib{Lindq}{book}{
  author={Lindqvist, Peter},
  title={Notes on the Stationary $p$-Laplace Equation},
  series={SpringerBriefs in Mathematics},
  publisher={Springer},
  address={Cham},
  date={2019},
  pages={xi+104},
  doi={10.1007/978-3-030-14501-9},
}

\bib{Mazya}{book}{
  author={Maz'ya, Vladimir},
  title={Sobolev Spaces},
  subtitle={With Applications to Elliptic Partial Differential Equations},
  edition={2},
  series={Grundlehren der mathematischen Wissenschaften},
  volume={342},
  publisher={Springer},
  address={Heidelberg},
  date={2011},
  pages={xxviii+866},
  doi={10.1007/978-3-642-15564-2},
}

\bib{pLrev}{article}{
  author={Mingione, Giuseppe},
  author={Palatucci, Giampiero},
  title={Developments and perspectives in nonlinear potential theory},
  journal={Nonlinear Analysis},
  volume={194},
  date={2020},
  pages={Article no. 111452, 17 pp.},
  doi={10.1016/j.na.2019.02.006},
}


\bib{EG92}{book}{
   author={Evans, Lawrence C.},
   author={Gariepy, Ronald F.},
   title={Measure theory and fine properties of functions},
   series={Studies in Advanced Mathematics},
   publisher={CRC Press, Boca Raton, FL},
   date={1992},
}


		
\bib{hkm}{book}{
   author={Heinonen, Juha},
   author={Kilpel\"{a}inen, Tero},
   author={Martio, Olli},
   title={Nonlinear potential theory of degenerate elliptic equations},
   publisher={Dover Publications Inc., Mineola, NY},
   date={2006},
}

\end{biblist}
\end{bibdiv}
\end{document}